\documentclass[12pt,reqno]{article}

\usepackage[usenames]{color}
\usepackage[colorlinks=true,
linkcolor=webgreen, filecolor=webbrown,
citecolor=webgreen]{hyperref}

\definecolor{webgreen}{rgb}{0,.5,0}
\definecolor{webbrown}{rgb}{.6,0,0}

\usepackage{amssymb}
\usepackage{graphicx}
\usepackage{amscd}
\usepackage{lscape}
\usepackage{tikz}
\usepackage{tikz-cd}
\usepackage{pgfplots}

\usetikzlibrary{matrix}
\usetikzlibrary{fit,shapes}
\usetikzlibrary{positioning, calc}
\tikzset{circle node/.style = {circle,inner sep=1pt,draw, fill=white},
        X node/.style = {fill=white, inner sep=1pt},
        dot node/.style = {circle, draw, inner sep=5pt}
        }
\usepackage{tkz-fct}

\usepackage{amsthm}
\newtheorem{theorem}{Theorem}

\newtheorem{proposition}[theorem]{Proposition}
\newtheorem{corollary}[theorem]{Corollary}
\newtheorem{conjecture}[theorem]{Conjecture}

\theoremstyle{definition}

\newtheorem{example}[theorem]{Example}

\usepackage{float}

\usepackage{graphics,amsmath}
\usepackage{amsfonts}
\usepackage{latexsym}
\usepackage{epsf}

\DeclareMathOperator{\sgn}{sgn}

\newcommand{\seqnum}[1]{\href{http://oeis.org/#1}{\underline{#1}}}

\begin{document}

\begin{center}
\vskip 1cm{\LARGE\bf Generalized Catalan recurrences, Riordan arrays, elliptic curves, and orthogonal polynomials} \vskip 1cm \large
Paul Barry\\
School of Science\\
Waterford Institute of Technology\\
Ireland\\
\href{mailto:pbarry@wit.ie}{\tt pbarry@wit.ie}
\end{center}
\vskip .2 in

\begin{abstract} We show that the Catalan-Schroeder convolution recurrences and their higher order generalizations can be solved using Riordan arrays and the Catalan numbers. We investigate the Hankel transforms of many of the recurrence solutions, and indicate that Somos $4$ sequences often arise. We exhibit relations between recurrences, Riordan arrays, elliptic curves and Somos $4$ sequences. We furthermore indicate how one can associate a family of orthogonal polynomials to a point on an elliptic curve, whose moments are related to recurrence solutions.\end{abstract}

This paper, which concerns generalized Catalan recurrences, their solutions using Riordan arrays, and applications to elliptic curve sequences, is arranged in the following sections.
\begin{enumerate}
\item A motivating example
\item Preliminaries
\item Generalized Catalan recurrences
\item A third order recurrence
\item A further recurrence
\item Conversion of parameters
\item From elliptic curve to recurrences and Somos sequences
\item The family $E_t : y^2+4xy+y=x^3+(t-1)x+tx$
\item The case of $E: y^2+axy+y=x^3+bx^2+cx$
\item Conclusions
\item Appendix: Orthogonal polynomials, Riordan arrays and the Hankel transform.
\end{enumerate}
\section{A motivating example}
The elliptic curve
$$E: y^2-xy+y=x^3-2x+x$$ passes through the point $P(0,0)$. The division polynomials of this curve, $\psi_n$, evaluated at the multiples $nP$ of the point $P(0,0)$ are given by
$$\psi_n(0,0)=(-1)^{\binom{n}{2}}F_n,$$ where $F_n$ is the $n$-th Fibonacci number \seqnum{A000045}. A consequence of this is that the coordinates of $nP$ are
$$x_n=x(nP)=\frac{F_n F_{n+2}}{F_{n+1}^2},$$ and
$$y_n=y(nP)=\frac{(-1)^n F_n}{F_{n+1}^3}.$$
The ratio $y_n/x_n$ is then given by
$$\frac{y_n}{x_n}=\frac{(-1)^n}{F_{n+1}F_{n+2}}.$$
We now form a generating function expressed as a continued fraction
$$\tilde{a}(x)=g_{E,P}(t)=\cfrac{1}{1+t+
\cfrac{x_1t^2}{1+\frac{y_1}{x_1}t+
\cfrac{x_2t^2}{1+\frac{y_2}{x_2}t+
\cfrac{x_3t^2}{1+\frac{y_3}{x_3}t+\cdots}}}}.$$
In this case, the generating function $g_{E,P}(t)$ expands to give the sequence $\tilde{a}_n$ that begins
$$1, -1, -1, 2, 2, -5, -5, 14, 14, -42, -42,\ldots.$$ The general term of this sequence is
$$\tilde{a}_n=(-1)^{\binom{n+1}{2}} C_{\lfloor \frac{n+1}{2} \rfloor}.$$ Here, we have
$C_n=\frac{1}{n+1} \binom{2n}{n}$ is the $n$-th Catalan number.

The Hankel transform $h_n=|a_{i+j}|_{0\le i,j \le n}$ of $\tilde{a}_n$ begins
$$1, -2, -3, 5, 8, -13, -21, 34, 55, -89, -144,\ldots,$$ with general term
$$h_n = (-1)^{\binom{n+1}{2}}F_{n+2}.$$
Thus we have $$h_n = (-1)^{n+1} \psi_{n+2}(0,0).$$

On the other hand, we can solve the equation $y^2-xy+y=x^3-2x+x$ for $y$.
We find that
$$y=\frac{(1-x)(\sqrt{1+4x}-1}{2}=(1-x)xc(x),\quad\text{or}\quad y=-\frac{(1-x)(1+\sqrt{1+4x})}{2}.$$
Here, $$c(x)=\frac{1-\sqrt{1-4x}}{2x}$$ is the generating function of the Catalan numbers.
The first solution expands to give the sequence that begins
$$0, 1, -2, 3, -7, 19, -56, 174, -561, 1859, -6292,$$ while the second solution begins
$$-1, 0, 2, -3, 7, -19, 56, -174, 561, -1859, 6292,\ldots.$$
We are interested in the common part of these sequences, namely the sequence that begins
$$2, -3, 7, -19, 56, -174, 561, -1859, 6292,\ldots.$$
This has its generating function given by
$$f(x)=\left(-\frac{(1-x)(1+\sqrt{1+4x})}{2}+1\right)/x^2=\frac{1+x-(1-x)\sqrt{1+4x}}{2x^2}.$$
We wish to work with a sequence with initial term $1$, which has essentially the same Hankel transform as this sequence. Thus we take the generating function
$$f_1(x)=\frac{1}{1-x-x^2 f(x)}=\frac{2}{1-3x+(1-x)\sqrt(1+4x)}=\frac{-(x-1)\sqrt{1+4x}-3x+1}{2x(x^2-4x+2)}.$$
Using the Fundamental Theorem of Riordan arrays, we can express this as
$$f_1(x)=\frac{1}{1-3x}c\left(\frac{-x(x^2-4x+2)}{(1-3x^2)^2}\right)=\left(\frac{1}{1-3x},\frac{-x(x^2-4x+2)}{(1-3x^2)^2}\right)\cdot c(x).$$
We now revert $xf_1(x)$, and divide the result by $x$, to obtain the generating function $a(x)$ given by
$$a(x)=\frac{1+3x+4x^2-(1+x)\sqrt{1+4x+8x^2}}{2x^3}.$$
This expands to give a sequence $a_n$ that begins
$$1, -1, -1, 8, -22, 33, 7, -212, 702, -1202, -58,\ldots.$$
This sequence satisfies the convolution recurrence (generalized Catalan recurrence)
$$a_n=-3 a_{n-1}-4 a_{n-2} + 2 a_{n-3} + \sum_{k=1}^{n-4} a_k a_{n-k-3},$$ with
$a_0=1, a_1=-1, a_2=-1, a_3=8$.
The Hankel transform of this sequence begins
$$1, -2, -3, 5, 8, -13, -21, 34, 55, -89, -144,\ldots.$$
Thus we now have a sequence $a_n$, constructed from the elliptic curve equation, that has the same Hankel transform as $\tilde{a}_n$. Moreover, we know the generating function of $a_n$. We can express $a(x)$ as
$$a(x)=\left(\frac{1+2x}{1+3x+4x^2},\frac{x^3(1+2x)}{(1+3x+4x^2)^2}\right)\cdot c(x).$$
It remains to relate $\tilde{a}(x)=g_{E,P}(x)$ to $a(x)$.

In this special case, the solution is quite easy to state.
We have
$$\tilde{a}_n = \sum_{k=0}^n \binom{n}{k}(-2)^{n-k} (-1)^k a_k.$$
In other words, the generating function $g_E(x)$ of $\tilde{a}_n=(-1)^{\binom{n+1}{2}} C_{\lfloor \frac{n+1}{2} \rfloor}$ is given by taking the second inverse binomial transform of the generating function of the sequence $(-1)^n a_n$.
We can verify this algebraically as follows. We take the generating function
$$a(x)=\frac{1+3x+4x^2-(1+x)\sqrt{1+4x+8x^2}}{2x^3}.$$
We now form the generating function of $(-1)^n a_n$ which is
$$a(-x)=-\frac{1-3x+4x^2-(1-x)\sqrt{1-4x+8x^2}}{2x^3}.$$ Taking the second inverse binomial transform now gives us
$$\frac{1}{1+2x}a\left(\frac{-x}{1+2x}\right)=\frac{(1+x)\sqrt{1+4x^2}-1-x-2x^2}{2x^3}.$$
We can verify independently that this last generating function is indeed the generating function of $\tilde{a}_n=(-1)^{\binom{n+1}{2}}C_{\lfloor \frac{n+1}{2} \rfloor}$.

We have
$$\frac{(1+x)\sqrt{1+4x^2}-1-x-2x^2}{2x^3}=
\cfrac{1}{1+x+
\cfrac{\frac{F_1 F_3}{F_2^2}x^2}{1-\cfrac{1}{F_2 F_3}x+
\cfrac{\frac{F_2 F_4}{F_3^2}x^2}{1+\cfrac{1}{F_3 F_4}x+
\cfrac{\frac{F_3 F_5}{F_4^2}x^2}{1-\cdots}}}}.$$

The sequence $\tilde{a}_n$ is the moment sequence (see the Appendix) of the family of orthogonal polynomials $P_n(t)$ defined by the three term recurrence
\begin{align*}P_n(t)&=(t+\frac{x_{n-1}}{y_{n-1}})P_{n-1}(t)+x_{n-1}P_{n-2}(t)\\
&=\left(t-\frac{(-1)^n}{F_n F_{n+1}}\right)P_{n-1}(t)+\frac{F_{n-1}F_{n+1}}{F_n^2} P_{n-2}(t),\end{align*}
with $P_0(t)=1, P_1(t)=t+1$.

The coefficient array of the family $P_n(t)$ then begins
$$\left(
\begin{array}{ccccccc}
 1 & 0 & 0 & 0 & 0 & 0 & 0 \\
 1 & 1 & 0 & 0 & 0 & 0 & 0 \\
 3/2 & 1/2 & 1 & 0 & 0 & 0 & 0 \\
1 & 7/3 & 2/3 & 1 & 0 & 0 & 0 \\
 8/5 & 7/5 & 17/5 & 3/5 & 1 & 0 & 0 \\
 1 & 31/8 & 17/8 & 35/8 & 5/8 & 1 & 0 \\
 21/13 & 31/13 & 95/13 & 35/13 & 70/13 & 8/13 & 1 \\
\end{array}
\right).$$

The inverse of this matrix, which is the moment matrix, begins
$$\left(
\begin{array}{ccccccc}
 1 & 0 & 0 & 0 & 0 & 0 & 0 \\
 -1 & 1 & 0 & 0 & 0 & 0 & 0 \\
 -1 & -1/2 & 1 & 0 & 0 & 0 & 0 \\
2 & -2 & -2/3 & 1 & 0 & 0 & 0 \\
 2 & 3/2 & -3 & -3/5 & 1 & 0 & 0 \\
 -5 & 5 & 8/3 & -4 & -5/8 & 1 & 0 \\
-5 & -9/2 & 9 & 3 & -5 & -8/13 & 1 \\
\end{array}
\right).$$
This exhibits the sequence $\tilde{a}_n$ as a moment sequence.

We note that the denominators in the coefficient matrix of the orthogonal polynomials are the Fibonacci numbers. Scaling up by these, we get the  integer matrix that begins as follows.

$$\left(
\begin{array}{ccccccc}
 1 & 0 & 0 & 0 & 0 & 0 & 0 \\
 1 & 1 & 0 & 0 & 0 & 0 & 0 \\
 3 & 1 & 2 & 0 & 0 & 0 & 0 \\
3 & 7 & 2 & 3 & 0 & 0 & 0 \\
 8 & 7 & 17 & 3 & 5 & 0 & 0 \\
 8 & 31 & 17 & 35 & 5 & 8 & 0 \\
 21 & 31 & 95 & 35 & 70 & 8 & 13 \\
\end{array}
\right).$$
The row sums of this matrix are given by $F_{n+1}F_{n+2}$.

The generating function of $\tilde{a}_n$ can be described using Riordan arrays as follows.
$$\tilde{a}(x)=\left(\frac{1}{1+x+2x^2}, \frac{-x^3}{(1+x+2x^2)^2}\right)\cdot c(x).$$
A consequence of this is that the sequence $\tilde{a}_n$ satisfies the convolution recurrence
$$\tilde{a}_n=-\tilde{a}_{n-1}-2 \tilde{a}_{n-2}-2 \tilde{a}_{n-3}-\sum_{k=1}^{n-4} \tilde{a}_k \tilde{a}_{n-k-3},$$ with
$a_0=1, a_1=-1, a_2=-1, a_3=2$.

\section{Preliminaries}

The product of two power series
$$a(x)=a_0+a_1 x + a_2 x^2+ a_3 x^3 + \cdots=\sum_{n=0}^{\infty}a_n x^n$$ and
$$b(x)=b_0+b_1 x + b_2 x^2+ b_3 x^3 + \cdots=\sum_{n=0}^{\infty} b_n x^n$$ is given by
$$a(x)b(x)=a_0 b_0 + (a_0 b_1+ a_1 b_0)x + (a_0 b_1+ a_1 b_1 + a_1 b_0) x^2 + \cdots.$$
That is,
$$a(x)b(x)=\sum_{n=0}^{\infty} \left(\sum_{k=0}^n a_k b_{n-k}\right)x^n.$$
The term $\sum_{k=0}^n a_k b_{n-k}$ is called the convolution of $a_0,a_1,\cdots,a_n$ with $b_0,b_1,\ldots,b_n$. Note that if we multiply the shifted sequences
$a_1+a_2 x+ a_3 x^2+ \cdots$ and $b_1+b_2 x+ b_3 x^2+\cdots$, then the product will begin
$$a_1 b_1 + (a_1 b_2+ a_2 b_1)x+ (a_1 b_3+a_2 b_2+a_3 b_1)x^2 + \cdots,$$ or
$$\sum_{n=0}^{\infty} \left(\sum_{k=1}^{n+1} a_{k} b_{n+2-k}\right)x^n.$$ This can be also be written as
$$\sum_{n=0}^{\infty} \left(\sum_{k=0}^n a_{k+1} b_{n+1-k}\right)x^n.$$
Expressions such as $\sum_{k=0}^n a_k b_{n-k}$ are known as convolutions, and equations involving terms of a sequence and such convolutions are known as convolution recurrences. For instance,
$$C_n = \sum_{k=0}^{n-1} C_k C_{n-1-k},$$ where we stipulate $C_0=1$, is a well known convolution recurrence.
Note that we have $C_1=C_0 C_0$, $C_2=C_0C_1+C_1C_0$, $C_3=C_0C_2+C_1C_1+C_2C_0$, and so on.

Starting from $C_0=1$, we obtain $C_0=1,1,2,5,14,42, \ldots$. These are the Catalan numbers \seqnum{A000108}. If we let $c(x)=\sum_{n=0}^{\infty}C_n x^n$, then we see that
$$C_1+C_2 x + C_3 x^2 + \cdots = C_0 C_0 + (C_0C_1+C_1C_0)x + \ldots = c(x) \times c(x)=c(x)^2.$$
Thus,
$$c(x)=C_0+C_1 x+ C_2 x^2+ \cdots=C_0+x(C_1+ C_2x+C_3x^2+ \ldots)=1+xc(x)^2.$$
This means that $c(x)$, the \emph{generating function} of the Catalan numbers, is a solution of the quadratic equation
$$u = 1+ x u^2, \quad \text{or} \quad x u^2-u+1 =0.$$ We obtain two solutions,
$$u(x)=\frac{1-\sqrt{1-4x}}{2x} \quad \text{and} \quad \frac{1+\sqrt{1-4x}}{2x}.$$
When $x=0$, the second solution leads to a division by zero, so we choose the first solution.
Thus $$c(x)=\frac{1-\sqrt{1-4x}}{2x}$$ is the generating function of the Catalan numbers. In this note, shall seek to solve more generalized convolution recurrences. The Catalan numbers $C_n$ and their generating function $c(x)$ will play an important role in this discussion. This will arise because we will often be solving quadratic equations, as above.
Thus if we consider a quadratic equation such as
$$ a u^2+ b u + c =0$$ then the solution
$$u = \frac{-b + \sqrt{b^2- 4 ac}}{2a}$$ can be expressed as follows. We have
\begin{align*}
\frac{-c}{b} c\left(\frac{ac}{b^2}\right)&=\frac{c}{b}\left(\frac{1-\sqrt{1-4\frac{ac}{b^2}}}{2 \frac{ac}{b^2}}\right)\\
&=\frac{-b}{2a}\left(1-\sqrt{1-4 \frac{ac}{b^2}}\right)\\
&=\frac{1}{2a} \left(-b + b \sqrt{1-4 \frac{ac}{b^2}}\right)\\
&=\frac{1}{2a}\left(-b+\sqrt{b^2-4 ac}\right).\end{align*}

The squared terms that give rise to the quadratic expressions to be solved will in turn be explained in part by the presence of the convolution elements. For instance, if we multiply a shifted power sequence $a_1+a_2 x + a_3 x^2+ \cdots$ by itself, thus getting a power series that begins
$$a_1 a_1 + (a_1 a_2+ a_2 a_1)x+ (a_1 a_3+ a_2 a_2+ a_3 a_1)x^2+ \cdots,$$ then the generating function of this expression will be
$$\frac{a(x)-1}{x} \cdot \frac{a(x)-1}{x}=\frac{(a(x)-1)^2}{x^2}=\frac{a(x)^2-2 a(x)+1}{x^2},$$ where
$$a(x)=\sum_{n=0}^{\infty} a_n x^n$$ is the generating function of the sequence $a_0,a_1,a_2,\ldots$.

The expression $\frac{-c}{b} c\left(\frac{ac}{b^2}\right)$ is reminiscent of the fundamental theorem of Riordan arrays, and in the cases that we shall consider, it will indeed be an instance of this result.  We shall therefore use Riordan arrays \cite{Book, Survey, SGWW} extensively. A Riordan array can be defined by a pair $(g, f)$ of power series
$$g(x)=g_0+g_1 x + g_2 x^2 + \cdots, \quad g_0 \ne 0,$$ and
$$f(x)=f_1 x + f_2 x^2 + f_3 x^3+\cdots,$$ with indeterminate $x$ and coefficients drawn from a suitable ring.
To this pair we can associate the lower-triangular matrix with general $(n,k)$-th term $T_{n,k}$ given by
$$T_{n,k}=[x^n] g(x)f(x)^k.$$
Here, $[x^n]$ is the functional on power series that extracts the coefficient of $x^n$ \cite{MC}.
Thus this matrix is the ``Riordan array'' associated with the pair $(g,f)$. In practice, we refer to either as a Riordan array. The Fundamental Theorem of Riordan Arrays asserts that the action of the pair $(g,f)$ acting on the generating function $h(x)$ as defined by
$$(g(x), f(x))\cdot h(x)=g(x)f(h(x))$$ is mirrored by multiplying the (infinite) vector whose elements are given by the expansion of $h(x)$ by the matrix $(T_{n,k})$.

When $f_1 \ne 0$ the matrices have a non-zero diagonal and hence they are invertible. If $f_1 =0$ we have arrays that are often described as ``stretched'', which are not invertible in the usual sense.  The Fundamental Theorem still holds.

Somos $4$ sequences are defined using the Weierstrass $\sigma$ function for appropriate elliptic curves. The link between this $\sigma$ function and the Hankel transform \cite{KrattL, Kratt, Layman} explains the importance of the Hankel transform in this note. The motivating example shows how the $x-$ and $y-$ coordinates of the multiples of special points on an elliptic curve can give rise to integer sequences whose Hankel transform is then the Somos $4$ sequence for that point. The link to the title of this note is that these sequences satisfy a recurrence of the desired type.

Where we encounter sequences that appear in the On-Line Encyclopedia of Integer Sequences \cite{SL1, SL2} we refer to them by their $Annnnnn$ number.

The $r$-th binomial transform $b_n$ of the sequence $a_n$ has general term $b_n=\sum_{k=0}^n \binom{n}{k}r^{n-k}a_k$. If the sequence $a_n$ has a generating sequence $g(x)$, then the sequence $b_n$ has generating function $\frac{1}{1-rx}g\left(\frac{x}{1-rx}\right)$.

The INVERT$(a)$ transform of the sequence $a_n$ whose generating function is $g(x)$ will have a generating function given by $\frac{g(x)}{1-axg(x)}=(g(x), xg(x))\cdot \frac{1}{1-ax}$.

If the generating function $g(x)$ of a sequence $a_n$ can be expressed as a Stieltjes continued fraction
$$g(x)=\cfrac{1}{1-\alpha_0 x -
\cfrac{\beta_1 x^2}{1-\alpha_1 x -
\cfrac{\beta_2 x^2}{1-\alpha_2 x- \cdots}}},$$ then the $r$-th binomial transform of $a_n$ will have a generating function given by
$$\cfrac{1}{1-(r+\alpha_0) x -
\cfrac{\beta_1 x^2}{1-(r+\alpha_1) x -
\cfrac{\beta_2 x^2}{1-(r+\alpha_2) x- \cdots}}}.$$
The INVERT$(a)$ transform of $a_n$ will have a generating function given by
$$\cfrac{1}{1-(a+\alpha_0) x -
\cfrac{\beta_1 x^2}{1-\alpha_1 x -
\cfrac{\beta_2 x^2}{1-\alpha_2 x- \cdots}}}.$$

The generating function $g(x)$ above will be designated by $\mathcal{J}(\alpha_0, \alpha_1,\ldots;\beta_1, \beta_2,\ldots)$.

If a sequence $a_0, a_1, \ldots$ with generating function $g(x)$ has a Hankel transform that begins $h_0, h_1, \ldots$ then the sequence with generating function $\frac{1}{1-x-x^2 g(x)}$ will have a Hankel transform that begins $1, h_0, h_1, \ldots$. In addition, this new sequence will have its initial term equal to $1$.

More information on orthogonal polynomials, Riordan arrays and Hankel transforms may be found in the Appendix.

\section{Generalized Catalan recurrences}
We let $C_n=\frac{1}{n+1}\binom{2n}{n}$ denote the $n$-th Catalan numbers \seqnum{A000108}. The sequence of Catalan numbers has generating function $c(x)=\frac{1-\sqrt{1-4x}}{2x}$. The Catalan numbers begin
$$1, 1, 2, 5, 14, 42, 132, 429, 1430, 4862, \ldots.$$
The large Schroeder numbers $S_n$ \seqnum{A006318} can then be defined by
$$S_n=\sum_{k=0}^n \binom{n+k}{2k}C_k.$$
These numbers begin
$$1, 2, 6, 22, 90, 394, 1806, 8558, 41586, 206098, \ldots.$$
Using the theory of Riordan arrays, we see that the generating function $S(x)$ of the large Schroeder numbers is given by an application of the Fundamental Theorem of Riordan arrays \cite{Book, SGWW}, which in this case says that
$$S(x)=\left(\frac{1}{1-x}, \frac{x}{(1-x)^2}\right)\cdot c(x)=\frac{1}{1-x} c\left(\frac{x}{(1-x)^2}\right).$$
This follows since the matrix representation of the Riordan array $\left(\frac{1}{1-x}, \frac{x}{(1-x)^2}\right)$ has general term
$$[x^n]\frac{1}{1-x}\frac{x^k}{(1-x)^{2k}}=\binom{n+k}{2k}.$$
Simplifying, we find that
$$S(x)=\frac{1-x-\sqrt{1-6x+x^2}}{2x}.$$
We now consider the generating function $\frac{c(x)+c(-x)}{2}=\frac{\sqrt{1+4x}+\sqrt{1-4x}}{2x}$, which expands to give the sequence
$$1, 0, 2, 0, 14, 0, 132, 0, 1430, 0,16796, \ldots,$$ which is the aerated sequence of Catalan numbers of even index \seqnum{A048990}.
We once again apply the Riordan array $\left(\frac{1}{1-x}, \frac{x}{(1-x)^2}\right)$ to this sequence, resulting in the sequence that begins
$$1, 1, 3, 11, 45, 197, 903, 4279, 20793, 103049, \ldots.$$
The generating function of this sequence, which is called the sequence of little Schroeder numbers $s_n$ \seqnum{A001003}, is then given by
$$\left(\frac{1}{1-x}, \frac{x}{(1-x)^2}\right)\cdot \frac{\sqrt{1+4x}+\sqrt{1-4x}}{2x}=\frac{1+x-\sqrt{1-6x+x^2}}{4x}.$$
We deduce that
$$s_n=\sum_{k=0}^n \binom{n+k}{2k}C_k \frac{1+(-1)^k}{2}=\sum_{k=0}^n \binom{n+2k}{4k}C_{2k}.$$
We can further use Riordan arrays to characterize these two sequences. The theory of Riordan arrays and orthogonal polynomials \cite{Book, classical, Barry_moments} allows us to derive the following result.
\begin{proposition} The large Schroeder numbers $S_n$ are the moments of the family of orthogonal polynomials whose coefficient array is given by the Riordan array
$$\left(\frac{1}{1+2x}, \frac{x}{1+3x+2x^2}\right)$$ and the little Schroeder numbers $s_n$ are the moments of the family of orthogonal polynomials whose coefficient array is given by the Riordan array
$$\left(\frac{1}{1+x}, \frac{x}{1+3x+2x^2}\right).$$
\end{proposition}
One consequence of this result is the following.
\begin{corollary} The Hankel transform $H_n=|S_{i+j}|_{0 \le i,j \le n}$ of the large Schroeder numbers, and the Hankel transform $h_n=|s_{i+j}|_{0 \le i,j \le n}$ of the little Schroeder numbers satisfy
$$H_n = h_n = 2^{\binom{n+1}{2}}.$$
\end{corollary}
Another corollary, which we can prove independently, is that the generating function $S(x)$ of the large Schroeder numbers has the following Jacobi continued fraction \cite{CFT, Wall} form:
$$S(x)=\cfrac{1}{1-2x-\cfrac{2x^2}{1-3x-\cfrac{2x^2}{1-3x-\cfrac{2x^2}{1-3x-\cdots}}}},$$ which we express as
$$S(x)=\mathcal{J}(2,3,3,3,\ldots;2,2,2,\ldots).$$
We solve for $v$ in the equation
$$v=\frac{1}{1-3x-2x^2 v},$$ then we verify that
$$S(x)=\frac{1}{1-2x-2x^2 v}.$$
We similarly have
$$s(x)=\mathcal{J}(1,3,3,3,\ldots;2,2,2,\ldots).$$
The Schroeder number sequences $S_n$ and $s_n$ each satisfy a simple convolution recurrence. For $S_n$, we have \cite{Qi}
$$S_n=3 S_{n-1}+\sum_{k=0}^{n-3} S_{k+1} S_{n-k-2},$$ with initial conditions
$S_0=1$, $S_1=2$. Note that we can write this as
$$S_n= 3 S_{n-1}+\sum_{k=1}^{n-2} S_k S_{n-k-1}.$$ The summation in the recurrence is then equal to
$$S_1 S_{n-2} + S_2 S_{n-3} + \cdots + S_{n-2}S_1.$$
For $s_n$, we have
$$s_n= 3 s_{n-1}+2\sum_{k=0}^{n-3} s_{k+1} s_{n-k-2},$$ with initial conditions
$s_0=1$, $s_1=1$.
These recurrences, and their generalizations, will be the subject of this note.
\begin{proposition}
We consider the generalized Catalan-Schroeder recurrence
$$a_n=s a_{n-1}+t \sum_{k=0}^{n-3} a_{k+1} a_{n-k-2},$$ with initial conditions
$a_0=1$, $a_1=p$. Then the generating function of the sequence $a_n$ is given by
$$\left(\frac{1+(p-s+t)x}{1-(s-2t)x}, \frac{t x(1+(p-s+t)}{(1-(s-2t)x)^2}\right)\cdot c(x).$$
\end{proposition}
Note that when $p=2$, $s=3$ and $t=1$, we obtain
$$\left(\frac{1+(2-3+1)x}{1-(3-2)x}, \frac{ x(1+(2-3+1)x)}{(1-(3-2)x)^2}\right)\cdot c(x)=\left(\frac{1}{1-x}, \frac{ x}{(1-x)^2}\right)\cdot c(x)=S(x),$$ while when $p=1$, $s=3$ and $t=2$, we obtain
$$\left(\frac{1+(1-3+2)x}{1+(2\cdot 2-3)x}, \frac{2 x(1+(1-3+2)x)}{(1+(2\cdot 2-3)x)^2}\right)\cdot c(x)=
\left(\frac{1}{1+x}, \frac{2 x}{(1+x)^2}\right)\cdot c(x)=s(x).$$
Note that this last result gives us that
$$s_n=\sum_{k=0}^n \binom{n+k}{2k}(-1)^{n-k}2^k C_k.$$
\begin{proof}
The terms $a_1a_{n-2}+\cdots+a_{n-2}a_1$, for $n \ge 3$, contribute successively the terms
$$a_1a_1,$$
$$a_1a_2+a_2a_1,$$
$$a_1a_3+a_2a_2+a_3a_1,$$
$$\cdots.$$
These correspond to the coefficients in the product
$$(a_1+a_2 x + a_3 x^2+ \cdots)(a_1+a_2x+ a_3 x^2+\cdots).$$
This product has generating function
$$\left(\frac{f(x)-1}{x}\right)^2=\frac{f(x)^2-2f(x)+1}{x^2}.$$
Thus we can convert the recurrence
$$a_n=s a_{n-1}+t \sum_{k=0}^{n-3} a_{k+1} a_{n-k-2}, \quad n\ge 2$$ into the following statement about the generating function $u(x)=\sum_{n=0}^{\infty} a_n x^n=1+px+ \sum_{n=2}^{\infty} a_n x^n$.
$$u=1+px+sx(u-1)+tx(u^2-2u+1).$$
Solving this, we get
$$u(x)=\frac{1-(s-2t)x-\sqrt{1-2sx+(s^2-4pt)x^2}}{2tx}.$$
This is equal to
$$\left(\frac{1+(p-s+t)x}{1-(s-2t)x}, \frac{tx(1+(p-s+t)x)}{1-(s-2t)x}\right)\cdot c(x).$$
\end{proof}
\begin{corollary} The generating function
$$\left(\frac{1+ax}{1+bx}, \frac{m x(1+ax)}{(1+bx)^2}\right)\cdot c(x)$$ expands to give the sequence solution of the convolution recurrence
$$a_n = (-b+2m) a_{n-1}+ m \sum_{k=0}^{n-3} a_{k+1}a_{n-k-3},$$ with $a_0=1, a_1=a-b+m$.
\end{corollary}
\begin{example} We calculate the general term $T_{n,k}$ of the Riordan array $\left(\frac{1+ax}{1+bx}, \frac{m x(1+ax)}{(1+bx)^2}\right)$.
We have
\begin{align*}
T_{n,k}&=[x^n] \frac{1+ax}{1+bx}\left(\frac{mx(1+ax)}{1+bx}\right)^k\\
&= m^k [x^{n-k}] \frac{(1+ax)^{k+1}}{(1+bx)^{2k+1}}\\
&= m^k [x^{n-k}] \sum_{j=0}^{k+1}\binom{k+1}{j} a^j x^j \sum_{i=0}^{\infty} \binom{-(2k+1)}{i} b^i x^i\\
&= m^k [x^{n-k}] \sum_{j=0}^{k+1}\binom{k+1}{j} a^j x^j \sum_{i=0}^{\infty} \binom{2k+1+i-1}{i}(-1)^i b^i x^i\\
&= m^k \sum_{j=0}^{k+1}\binom{k+1}{j} a^j \binom{2k+n-k-j}{n-k-j}(-b)^{n-k-j}\\
&= m^k \sum_{j=0}^{k+1} \binom{k+1}{j}\binom{n+k-j}{2k}a^j (-b)^{n-k-j}.\end{align*}
Thus the expansion of  $\frac{1+ax}{1+bx}c\left(\frac{mx(1+ax)}{1+bx}\right)$ has general term given by
$$a_n=\sum_{k=0}^n (\sum_{j=0}^{k+1} \binom{k+1}{j}\binom{n+k-j}{2k}a^j (-b)^{n-k-j})m^k C_k.$$
Note that in practice care has to be taken with the terms $T_{0,k}$ using the above formula so we have the following correction (for $a \ne 0$).
$$a_n=\sum_{k=0}^n (\sum_{j=0}^{k+1} \binom{k+1}{j}\binom{n+k-j}{2k}a^j (-b)^{n-k-j})m^k C_k+\frac{b}{a}0^k.$$
\end{example}
We denote the general solution of the
 recurrence $$a_n=s a_{n-1}+t \sum_{k=0}^{n-3} a_{k+1} a_{n-k-2},$$ with initial conditions
$a_0=1$, $a_1=p$ by $a_n(p,s,t)$.
\begin{example} When $(p,s,t)=(1,1,1)$, we get the sequence $a_n(1,1,1)$ that begins
$$1,1, 1, 2, 4, 9, 21, 51, 127, 323, 835, 2188, 5798, \ldots.$$
These are the Motzkin numbers, with a $1$ prepended.
The Hankel transform of this sequence has generating function $\frac{1-x}{1-x+x^2}$, beginning
$$1, 0, -1, -1, 0, 1, 1, 0, -1,\ldots.$$
The Hankel transform of the shifted sequence (the Motzkin numbers) is known to be the given by the all $1$'s sequence $1,1,1,1,\ldots$. The generating function of the Motzkin numbers is $\mathcal{J}(1,1,1,\ldots;1,1,1,\ldots)$.
\end{example}
\begin{example} When $(p,s,t)=(1,1,2)$, we get the sequence $a_n(1,1,2)$ that begins
$$1,1, 1, 3, 7, 21, 61, 191, 603, 1961, 6457, 21595,\ldots.$$ This is the sequence \seqnum{A025235} with a prepended $1$. The sequence \seqnum{A025235} counts Motzkin paths with the up step in two colors. The Hankel transform of $a_n(1,1,2)$ begins
$$1, 0, -4, -16, 128, 6144, 65536, -20971520, -3758096384,\ldots$$
while that of the shifted sequence $u_{n+1}(1,1,2)$ or \seqnum{A025235} is $2^{\binom{n+1}{2}}$. This follows from the fact that the generating function of this latter sequence is $\mathcal{J}(1,1,1,\ldots;2,2,2,\ldots)$.
\end{example}
\begin{example} When $(p,s,t)=(1,2,1)$, the sequence $a_n(1,2,1)$ that we get is the sequence of Catalan numbers $C_n$. This corresponds to a bijection between $2$-colored Motzkin paths (where the level step can have one of two colors) of length $n-1$ and Dyck paths.  It is a classical result that the Hankel transforms of this sequence and its shift $C_{n+1}$ are the all $1$'s sequence.
\end{example}
With regard to the shifted sequence $a_{n+1}(p,s,t)$, we have the following result.
\begin{proposition}
The sequence $\frac{1}{p} a_{n+1}(p,s,t)$ has generating function $\mathcal{J}(s,s,s,\ldots; pt,pt,pt,\ldots)$, and Hankel transform $(pt)^{\binom{n+1}{2}}$ (for $p \ne 0$).
\end{proposition}
\begin{proof} The sequence with generating function $\mathcal{J}(s,s,s,\ldots; pt,pt,pt,\ldots)$ has a generating function $g(x)$ that satisfies
$$g(x)=\frac{1}{1-s x - pt x^2 g(x)}.$$
Solving this equation gives us
$$g(x)=\frac{1-sx-\sqrt{1-2sx+(s^2-4pt)x^2}}{2ptx^2}.$$ We see that this is the generating function of $\frac{1}{p} a_{n+1}(p,s,t)$.
\end{proof}
\begin{corollary} The sequence $\frac{1}{p} a_{n+1}(p,s,t)$ is the moment sequence for the family of orthogonal polynomials whose coefficient array is given by the Riordan array
$$\left(\frac{1}{1+sx+pt x^2}, \frac{x}{1+sx + pt x^2}\right).$$
\end{corollary}
\begin{corollary} We have the moment representation
$$\frac{1}{p} a_{n+1}(p,s,t)=\frac{1}{2 \pi} \int_{s-2 \sqrt{pt}}^{s+2 \sqrt{pt}} x^n \frac{\sqrt{-x^2+2sx+4pt-s^2}}{pt}\,dt.$$
\end{corollary}
\begin{proof} We apply the Stieltjes-Perron transform \cite{Henrici} to the generating function of $a_{n+1}$.
\end{proof}
\begin{corollary} If $s = 0$ and $t \ne 0$,  we have
$$a_n(p,0,t)=\frac{1}{2 \pi} \int_{-2\sqrt{pt}}^{2\sqrt{pt}} x^n \frac{\sqrt{4pt-x^2}}{tx}\,dx+0^n.$$
\end{corollary}
\begin{corollary} If $s\ne 0$ and $t \ne 0$, we have
$$a_n(p,s,t)=\frac{1}{2 \pi} \int_{s-2\sqrt{pt}}^{s+2\sqrt{pt}} x^n \frac{\sqrt{-x^2+2sx+4pt-s^2}}{tx}\,dx+0^n\left(1-\frac{s}{2t}-\sgn(4pt-s^2)\frac{\sqrt{s^2-4pt}}{2t}\right).$$
\end{corollary}
We can formulate the following conjecture concerning the Hankel transform of the sequence $a_n(p,s,t)$.
\begin{conjecture} The Hankel transform $h_n(p,s,t)$ of the sequence $a_n(p,s,t)$ is given by
\begin{align*}h_n(p,s,t)&=t^{\binom{n}{2}} p^{\binom{n+1}{2}} [x^n] \frac{1-px}{1-sx+ptx^2}\\
&=t^{\binom{n}{2}} p^{\binom{n+1}{2}}\left(\sum_{i=0}^n \binom{i}{n-i}(-pt)^{n-i}s^{2i-n}-p \sum_{i=0}^{n-1}\binom{i}{n-i-1}(-pt)^{n-i-1}s^{2i-n+1}\right).\end{align*}
\end{conjecture}
We note that we have
$$\left(\frac{1-x}{1+ptx^2}, \frac{x}{1+ptx^2}\right)\cdot \frac{1}{1-sx}=\frac{1-px}{1-sx+ptx^2}.$$
Thus the Hankel transform $h_n(p,s,t)$ is given by the evaluation of a sequence of scaled orthogonal polynomials, since the Riordan array $\left(\frac{1-x}{1+ptx^2}, \frac{x}{1+ptx^2}\right)$ is the coefficient matrix of a family of orthogonal polynomials \cite{Barry_moments}

We have the following conjecture concerning the Hankel transforms $h_n(p,s,t)$ and Somos $4$ sequences. We recall that a $(\alpha, \beta)$ Somos $4$ sequence is a sequence $e_n$ that satisfies
$$ e_n=\frac{ \alpha e_{n-1} e_{n-3} + \beta e_{n-2}^2}{e_{n-4}}$$ for given $e_0, e_1, e_2, e_3$. Somos $4$ sequences are closely related to elliptic curves \cite{Hone, Yura}.
\begin{conjecture} The Hankel transform $h_n(p,s,t)$ of the generalized Schroeder numbers $a_n(p,s,t)$ is a $((pst)^2, (pt)^3(pt-s^2))$ Somos $4$ sequence.
\end{conjecture}
\begin{example} We consider the sequence $a_n(1,-1,-1)$ which begins
$$1, 1, -1, 0, 2, -3, -1, 11, -15, -13, 77,\ldots.$$ Its Hankel transform is the signed Fibonacci sequence
$$1, -2, -3, 5, 8, -13,\ldots$$ with general term $(-1)^{\binom{n}{2}}F_{n+1}$. This is a (trivial) $(1,2)$ Somos $4$ sequence.
\end{example}
\begin{example} The sequence $a_n(2,1,1)$ begins
$$1, 2, 2, 6, 14, 42, 122, 382, 1206, 3922, 12914,\ldots.$$ This is \seqnum{A014431}. Its Hankel transform begins $$1, -2, -24, -64, 5120, 229376,\ldots.$$
This is a $(4,8)$ Somos $4$ sequence.
\end{example}
\section{A third order recurrence}
We begin this section by looking at the generalized Catalan recurrence
$$a_n= r a_{n-1}+s a_{n-2} + t \sum_{k=1}^{n-3} a_k a_{n-k-2},$$ with
$a_0=1$, $a_1=p$ and $a_2=q$.
This may equivalently be written as
$$a_n= r a_{n-1}+s a_{n-2} + t \sum_{i=0}^{n-4} a_{i+1} a_{n-i-3},$$ with
$a_0=1$, $a_1=p$ and $a_2=q$. Thus in the summation part, the sum is
$$a_1 a_{n-3}+\cdots + a_{n-3}a_1.$$
We shall denote the solution of this recurrence by $a_n(p,q,r,s,t)$. This sequence begins
$$1, p, q, ps + qr, p^2t + prs + q(r^2 + s), p^2rt + p(2qt + r^2s + s^2) + qr(r^2 + 2s),\ldots.$$ We let $u(x)=u(x;p,q,r,s,t)$ be the generating function of this solution sequence.
We have the following proposition.
\begin{proposition} The generating function $u(x)$ is given by
$$\left(\frac{1-(r-p)x-(-q+pr+s-t)x^2}{1-rx-(s-2t)x^2}, \frac{tx^2(1-(r-p)x-(-q+pr+s-t)x^2)}{(1-rx-(s-2t)x^2)^2}\right)\cdot c(x).$$
\end{proposition}
\begin{proof}
We translate the recurrence (for $n \ge 3$) and the initial conditions into the following equation for the generating function $u(x)=\sum_{n=0}^{\infty} a_n(p,q,r,s,t)x^n$.
$$u(x)=1+px+qx^2+rx(u-1-xp)+sx^2(u-1)+tx^2(u^2-2u+1).$$
Then the solution is given by
$$u(x)=\frac{1-rx-(s-2t)x^2-\sqrt{1-2rx+(r^2-2s)x^2+2(rs-2pt)x^3+(s^2-4qt+4prt)x^4}}{2tx^2}.$$
This is equal to
$$\left(\frac{1-(r-p)x-(-q+pr+s-t)x^2}{1-rx-(s-2t)x^2}, \frac{tx^2(1-(r-p)x-(-q+pr+s-t)x^2}{(1-rx-(s-2t)x^2)^2}\right)\cdot c(x).$$
\end{proof}
\begin{example} The sequence $a_n(1,2,2,1,1)$ begins
$$1, 1, 2, 5, 13, 35, 97, 275, 794, 2327, 6905,\ldots.$$ This is \seqnum{A086581}, which counts the number of Dyck paths of semi-length $n$ that avoid DDUU. In this case we have
$$a_n=\sum_{k=0}^{\lfloor \frac{n}{2} \rfloor} \binom{n+k}{3k}C_k.$$
The generating function is
$$\frac{1}{1-x}c\left(\frac{x^2}{(1-x)^3}\right).$$
\end{example}
\begin{example} The sequence $a_n(1,2,1,2,1)$ coincides with the Motzkin numbers $M_n$. The generating function is
$$\frac{1}{1-x}c\left(\frac{x^2}{(1-x)^2}\right),$$ so that
$$M_n = \sum_{k=0}^{\lfloor \frac{n}{2} \rfloor} \binom{n}{2k}C_k.$$
\end{example}
\begin{example} The sequence $a_n(1,2,2,2,1)$ begins
$$1, 1, 2, 6, 17, 50, 150, 458, 1420, 4460, 14165,\ldots.$$ This is \seqnum{A025272}$(n+1)$. It has as generating function the power series given by
$$\frac{1-x-x^2}{1-2x}c\left(\frac{x^2(1-x-x^2)}{(1-2x)^2}\right).$$ The Hankel transform of the shifted sequence $a_{n+1}$ in this case begins
$$1, 2, 3, -5, -28, -67, -411, -506, 10855, 74231, 664776,\ldots.$$ This is the $(1,-2)$ Somos $4$ sequence \seqnum{A178376}$(n+1)$, which is defined by the elliptic curve $ y^2 +y = x^3 +3x^2 +x$.
\end{example}
\begin{example} The sequence $-a_{n+1}(-1,2,-2,-1,-1)$ begins
$$1, -2, 3, -3, -1, 15, -47, 98, -133, 17, 579,\ldots,$$ It has its generating function given by
 $$\frac{1+x}{1+2x-x^2}c\left(\frac{-x^2(1+x)}{(1+2x-x^2)^2}\right).$$ Its Hankel transform begins
$$1, -1, 1, 2, -1, -3,\ldots$$ which is a $(1,1)$ Somos $4$ sequence. This is a variant of \seqnum{A006769}, which is the elliptic divisibility sequence associated with elliptic curve $E: y^2 + y = x^3 - x$ and multiples of the point $(0,0)$.  In this case, the sequence $a_n(-1,2,-2,-1,-1)$ begins
$$1, -1, 2, -3, 3, 1, -15, 47, -98, 133, -17, -579,\ldots$$ and its Hankel transform is the $(1,1)$-Somos sequence beginning
$$1, 1, -2, -1, 3, -5,\ldots,$$ which is $(-1)^{n+1}$\seqnum{A006769}$(n+3)$.
\end{example}

\begin{example}
The sequence $a_n(-1,-2,2,-1,-1)$ begins
$$1, -1, -2, -3, -5, -11, -27, -65, -154, -371, -917, -2303,\ldots.$$ Its generating function is given by
$$\frac{1-3x}{1-2x-x^2}c\left(\frac{x^2(1-3x)}{(1-2x-x^2)^2}\right).$$ Its Hankel transform begins
$$1, -3, 2, 11, -29, -21,\ldots$$ which is a $(1,1)$ Somos $4$ sequence. It is a variant of \seqnum{A178384}, which is associated to the elliptic curve $y^2 + y = x^3 + x$. The shifted sequence $-a_{n+1}$ in this case has a Hankel transform that begins $$1, -1, -3, -2, 11, 29,\ldots.$$
\end{example}
We have the following two conjectures regarding these sequences, their Hankel transforms, and Somos $4$ sequences. The first conjecture claims that in all cases, the sequence $\frac{1}{p} a_{n+1}(p,q,r,s,t)$ has a Somos $4$ Hankel transform.
\begin{conjecture} The Hankel transform of the sequence $\frac{1}{p} a_{n+1}(p,q,r,s,t)$ is a $((pt)^2, -t^2(p^2s+pqr-q^2))$ Somos $4$ sequence.
\end{conjecture}
\begin{conjecture} If $r=1$, $t=1$ and $s=q-p+1$, then the sequence $a_n(p,q,r,s,t)=a_n(p,q,1,q-p+1,1)$ has a Hankel transform that is a $(p^2, p^3-pq+q^2-p^2(1+q))$ Somos $4$ sequence.
\end{conjecture}
We note that in this last case, the generating function of the sequence $a_n$ will be given by
$$\left(\frac{1-(1-p)x}{1-x+(p+q-1)x^2}, \frac{x^2(1-(1-p)x)}{(1-x+(p+q-1)x^2)^2}\right)\cdot c(x).$$
In another direction, it is interesting to consider sequences with generating functions of the form
$$\frac{1-x-\alpha x^2}{1-x-x^2}c\left(\frac{x^2(1-x-\alpha x^2)}{(1-x-x^2)^2}\right).$$
For these, we have the following conjecture.
\begin{conjecture}
The Hankel transform $h_n$ of the sequence with generating function $$\frac{1-x-\alpha x^2}{1-x-x^2}c\left(\frac{x^2(1-x-\alpha x^2)}{(1-x-x^2)^2}\right)$$ is given by
$$h_n=(2-\alpha)^{\lfloor \frac{(n+1)^2}{4} \rfloor} [x^n] \frac{(1+x)(1+(\alpha-2)x^2)}{1-3x^2-(\alpha-2)x^4}.$$
\end{conjecture}
Note that
$$\frac{(1+x)(1+(\alpha-2)x^2)}{1-3x^2-(\alpha-2)x^4}=(1+x)f(x^2;\alpha),$$ where
$$f(x;\alpha)=\frac{1 + (\alpha - 2) x} {1 - 3 x - (\alpha - 2) x^2}.$$
For instance, when $\alpha=1$, we get the generating function $\frac{1-x}{1-3x+x^2}$ of the bisected Fibonacci numbers $F_{2n+1}$ $1,2,5,13,34, 89,\ldots$ \seqnum{A001519}$(n+1)$.
In general, if we set
\begin{align*}d_n&=\sum_{k=0}^n \binom{k}{n-k}(\alpha-2)^{n-k}3^{2k-n}+(\alpha-2)\sum_{k=0}^n \binom{k}{n-k-1}(\alpha-2)^{n-k-1}3^{2k-n+1}\\
&=\sum_{k=0}^n \left(\binom{k}{n-k}+3 \binom{k}{n-k-1}\right)(\alpha-2)^{n-k}3^{2k-n},\end{align*}
then we have
$$h_n =(2-\alpha)^{\lfloor \frac{(n+1)^2}{4} \rfloor} d_{\lfloor \frac{n}{2} \rfloor}.$$
Note that here $\binom{k}{n-k}+3 \binom{k}{n-k-1}$ is the general term of the Riordan array $(1+3x, x(1+x))$.
In fact, we have
$$f(x;\alpha)=\left(1+(\alpha-2)x, x\left(1+ \frac{\alpha-2}{3}x)\right)\right)\cdot \frac{1}{1-3x}.$$
\begin{example} We take $\alpha=1$, to get the equation
$$\left(1-x, x\left(1-\frac{x}{3}\right)\right)\cdot \frac{1}{1-3x}=\frac{1-x}{1-3x+x^2}=\sum_{n=0}^{\infty}F_{2n+1}x^n,$$ which is the generating function of the Fibonacci bisection $F_{2n+1}$ that begins $1,2,5,13,34,\ldots$.
We have $2-\alpha=1$. In this case the sequence we get begins
$$1, 0, 1, 1, 4, 7, 20, 43, 112, 263, 669, 1640, 4166,\ldots$$ with generating function
$$c\left(\frac{x^2}{1-x-x^2}\right).$$
The Hankel transform is then given by $F_{2 \lfloor \frac{n}{2}\rfloor +1}$, or
$$1,1,2,2,5,5,13,13,34,34,\ldots.$$
\end{example}
\begin{example} We take $\alpha=-1$. Then $2-\alpha=3$.
The sequence in question has generating function
$$\left(\frac{1-x+x^2}{1-x-x^2}, \frac{x^2(1-x+x^2}{(1-x-x^2)^2}\right)\cdot c(x).$$
We obtain the sequence that begins
$$1, 0, 3, 3, 12, 21, 66, 147, 426, 1065, 3009, 7986, 22476,\ldots.$$
The Hankel transform begins
$$1, 3, 0, 0, -2187, -59049, -4782969,\ldots.$$ Dividing this by $3^{\lfloor \frac{(n+1)^2}{4} \rfloor}$, we obtain
$$1, 1, 0, 0, -3, -3, -9,\ldots$$ which is the expansion of $(1+x)(1-3x^2)/(1-3x^2+3x^4)$. This is the doubling of the sequence
$$1, 0, -3, -9, -18, -27, -27, 0, 81, 243, 486,\ldots$$ which has generating function $\frac{1-3x}{1-3x+3x^2}$.
\end{example}
\begin{example} The case of $\alpha=2$ is of special interest. We have $2-\alpha=0$, so that the Hankel transform is just the sequence $1,0,0,0,\ldots$. This is explained by the fact that
$$\left(\frac{1-x-2x^2}{1-x-x^2}, \frac{x^2(1-x-2x^2)}{(1-x-x^2)^2}\right)\cdot c(x)=\frac{1-x-2x^2}{1-x-x^2}c\left(\frac{x^2(1-x-2x^2)}{(1-x-x^2)^2}\right)=1.$$
\end{example}
A choice of parameters $(p,q,r,s,t)$ that satisfies this context is given by $(p,q,r,s,t)=(0,2-\alpha, 1,3,1)$. Thus the recurrence
$$a_n = a_{n-1} + 3 a_{n-2} + \sum_{k=1}^{n-3} a_k a_{n-k-2},$$ with $a_0=1, a_1=0, a_2=\alpha-2$ will have a solution with generating function given by $\frac{1-x-\alpha x^2}{1-x-x^2}c\left(\frac{x^2(1-x-\alpha x^2)}{(1-x-x^2)^2}\right)$.
We can reformulate our conjecture in the following manner.
\begin{conjecture} The solution $a_n(\beta)$ of the recurrence
$$a_n = a_{n-1} + 3 a_{n-2} + \sum_{k=1}^{n-3} a_k a_{n-k-2},$$ with $a_0=1, a_1=0, a_2=\beta$ has a Hankel transform $h_n(\beta)$ given by
$$h_n(\beta)=\beta^{\lfloor \frac{(n+1)^2}{4} \rfloor} [x^n] \frac{(1+x)(1-\beta x^2)}{1-3x^2+\beta x^4}.$$
\end{conjecture}

We can convert from the recurrence to the Riordan array and vice versa.
Thus given $(p,q,r,s,t)$, we obtain $(a,b,c,d,m)$ as follows.
\begin{align*}
a&=p-r\\
b&=q-pr-s+t\\
c&=-r\\
d&=-s+2t\\
m&=t \end{align*}
Given $(a,b,c,d,m)$, we obtain $(p,q,r,s,t)$ as follows.
\begin{align*}
p&=a-c\\
q&=-ac+b+c^2-d+m\\
r&=-c\\
s&=-d+2m\\
t&=m. \end{align*}

A related result is the following.
\begin{proposition} Assume that $a_n$ satisfies the following convolution recursion relation:
$$a_n=
\begin{cases}
1, \quad\quad \text{if}\quad n=0;\\
r, \quad\quad \text{if}\quad n=1;\\
r a_{n-1}+ s a_{n-2} + t \sum_{i=0}^{n-2} a_i a_{n-2-i},\quad \text{if}\quad n>1.
\end{cases}$$
Then the sequence $a_n$ has generating function
$$\left(\frac{1}{1-rx-sx^2}, \frac{tx^2}{(1-rx-sx^2)^2}\right)\cdot c(x),$$ and the Hankel transform of $a_n$ is a $((pt)^2, t^2(t+s)^2-r^2t^3)$ Somos $4$ sequence.
\end{proposition}
Note that $\sum_{i=0}^{n-2} a_i a_{n-2-i}$ expands to give $(a_0 a_{n-2}+ \cdots + a_{n-2}a_0)$.
\begin{proof} For a proof of the Hankel transform assertion, see \cite{Conj}.
We translate the recurrence to the equality
$$u(x) = 1+ r x u(x)+ sx^2u(x)+ t x^2 u(x)^2.$$ Solving for $u(x)$ now gives the result.
\end{proof}
Examples of such sequences are given in the following table.
\begin{center}
\begin{tabular}{|c|l|c|}
\hline $(r,s,t)$ & Annnnnn & Description\\
\hline $(1,1,1)$ & \seqnum{A128720} & Skew Dyck paths avoiding UUU \\
\hline $(2,1,1)$ & \seqnum{A085139}$(n+1)$ & G.f. is $\mathcal{J}(2,1,2,2,1,2,2,1,2,\dots;2,2,1,2,2,1,\ldots)$ \\
\hline $(1,2,1)$ & \seqnum{A174171} & Chebyshev transform of Motzkin numbers \\
\hline $(2,2,1)$ & \seqnum{A174403} & Hankel transform is \seqnum{A174404} \\
\hline $(3,1,1)$ & \seqnum{A084782}$(n+1)$ & \seqnum{A084782}$(n)=\sum_{j=0}^n \sum_{i=0}^j a_i a_{j-i}F_{n-j}$\\
\hline $(1,2,-1)$ &\seqnum{A187256} & Peakless Motzkin paths, with level steps in two colors\\ \hline
\end{tabular}
\end{center}
The sequence $a_n$ begins
$$1, r, r^2 + s + t, r(r^2 + 2s + 3t), r^4 + r^2(3s + 6t) + s^2 + 3st + 2t^2,\ldots.$$
The sequence with prepended $1$ that begins $1,1,r,r^2+s+t,\ldots$ then has generating function
$$\left(\frac{1-(r-t-1)x-sx^2}{1-(r-2t)x-sx^2},\frac{x(1-(r-t-1)x-sx^2}{(1-(r-2t)x-sx^2)^2}\right)\cdot c(x).$$
When $t=1$, we get
$$\left(1,\frac{x}{1-(r-2)x-sx^2}\right)\cdot c(x).$$
In fact, we can modify the above proposition as follows.
\begin{proposition} Assume that $a_n$ satisfies the following convolution recursion relation:
$$a_n=
\begin{cases}
1, \quad\quad \text{if}\quad n=0;\\
p, \quad\quad \text{if}\quad n=1;\\
r a_{n-1}+ s a_{n-2} + t \sum_{i=0}^{n-2} a_i a_{n-2-i}, \quad\text{if}\quad n>1.
\end{cases}$$
Then the sequence $a_n$ has generating function
$$\left(\frac{1+(p-r)x}{1-rx-sx^2}, \frac{tx^2(1+(p-r)x}{(1-rx-sx^2)^2}\right)\cdot c(x),$$ and the Hankel transform of $a_n$ is a $((pt)^2, t^2((t+s)^2-p^2s+prs+prt-2p^2t)$ Somos $4$ sequence.
\end{proposition}

\section{A further recurrence}
We now consider the recurrence
$$a_n=u a_{n-1}+ v a_{n-2}+ w a_{n-3} + t \sum_{k=1}^{n-4} a_k a_{n-k-3},$$ with
$a_0=1$, $a_1=p$, $a_2=q$, and $a_3=s$. We note that the term $\sum_{k=1}^{n-4} a_k a_{n-k-3}$ here expands to give
$(a_1 a_{n-4} + \cdots + a_{n-4}a_1)$.
We denote the solution of this recurrence by $a_n(p,q,s,u,v,w,t)$. We then have the following result.
\begin{proposition} The generating function of the sequence $a_n(p,q,s,u,v,w,t)$ is given by
$$(\frac{1+(p-u)x-(v+pu-q)x^2-(w-s-t+qu+pv)x^3}{1-ux-vx^2-(w-2t)x^3},$$
$$\quad\quad\frac{tx^3(1+(p-u)x-(v+pu-q)x^2-(w-s-t+qu+pv)x^3)}{(1-ux-vx^2-(w-2t)x^3)^2})\cdot c(x).$$
\end{proposition}
\begin{example} The sequence $a_n(1,1,2,1,1,1,1)$ begins
$$1, 1, 1, 2, 4, 8, 16, 33, 69, 146, 312, 673, 1463,\ldots.$$
This is \seqnum{A004149}$(n+1)$, which gives the number of Motzkin paths of length $n$ with no peaks or valleys.
The generating function of this sequence is given by
$$\frac{1}{1-x} c\left(\frac{x^3}{(1-x^2)(1-x)^2}\right).$$
\end{example}
\begin{example} The sequence $a_n(1,2,3,1,1,1,1)$ which begins
$$1, 1, 2, 3, 6, 12, 25, 53, 114, 249, 550, 1227, 276,\dots$$
 is \seqnum{A162985}. It counts the number of Dyck paths of semi-length $n$ avoiding UUU, DDD and UUDUDD (Deutsch).
Its generating function is given by
$$\frac{1}{1-x-x^2+x^3}c\left(\frac{x^3}{(1-x-x^2+x^3)^2}\right).$$
Note that this generating function can be expressed as the continued fraction
$$\cfrac{1}{1-x-x^2+x^3-\cfrac{x^3}{1-x-x^2+x^3-\cfrac{x^3}{1-x-x^2+x^3\cdots}}}.$$
\end{example}
\begin{example} The sequence $a_n(1,1,2,1,0,2,1)$ which begins
$$1, 1, 1, 2, 4, 7, 13, 26, 52, 104, 212, 438, 910, 1903, 4009, 8494,\ldots$$ is \seqnum{A023431}. It counts the number of Motzkin paths of length $n$ with no peaks and no double rises. Its generating function is given by
$$g(x)=\left(\frac{1}{1-x}, \frac{x^3}{(1-x)^2}\right)\cdot c(x).$$ We then have
$$a_n(1,1,2,1,0,2,1)=\sum_{k=0}^n \binom{n-k}{2k}C_k.$$
It is interesting to note that the Riordan array $(g(x), x g(x))$ is a pseudo-involution in the Riordan group \cite{PS}.
More generally, if
$$g(x;\alpha)=\left(\frac{1}{1-\alpha x}, \frac{x^3}{(1-\alpha x)^2}\right)\cdot c(x)$$ then the Riordan array
$(g(x;\alpha), xg(x;\alpha))$ is a pseudo-involution. For example, $g(x;2)$ is the generating function of $a_n(2,4,9,2,0,2,1)$. We will then have
$$a_n(2,4,9,2,0,2,1)=\sum_{k=0}^{\lfloor \frac{n}{3} \rfloor} \binom{n-k}{2k}2^{n-3k}C_k.$$
This sequence is \seqnum{A091561} \cite{PS}. Its Hankel transform is the sequence
$$1, 0, -1, -1, 0, 1, 1, 0, -1, -1, 0,\ldots$$ with generating function $\frac{1-x}{1-x+x^2}$.
\end{example}
\begin{example} The sequence $a_n(1,2,4,1,1,2,1)$ is the RNA sequence that begins
$$1,1, 2, 4, 8, 17, 37, 82, 185, 423, 978, 2283, 5373,\ldots.$$ This is \seqnum{A004148}, with generating function $$\frac{1}{1-x-x^2} c\left(\frac{x^3}{(1-x-x^2)^2}\right).$$
The related sequence with generating function
$$\frac{1-x}{1-x-x^2} c\left(\frac{x^3(1-x)}{(1-x-x^2)^2}\right)$$ begins
$$1, 0, 1, 2, 3, 7, 14, 28, 60, 126, 268, 579, 1253,\ldots.$$
Its Hankel transform begins
$$1, 1, -2, -3, -7, 5, 32, 83, 87, -821, -2366,\ldots.$$
This is a $(1,-1)$ Somos $4$ sequence.
\end{example}
\begin{example} The sequence $a_n(1,3,6,1,2,2,1)$ begins
$$1, 1, 3, 6, 14, 33, 79, 194, 482, 1214, 3090, 7936,\ldots.$$
This sequence has its generating function given by
$$\frac{1}{1-x-2x^2}c\left(\frac{x^3}{(1-x-2x^2)^2}\right).$$
Its Hankel transform begins
$$1, 2, 1, -7, -16, -57, -113, 670, 3983, 23647,\ldots.$$
This is a $(1,-2)$ Somos $4$ sequence.
\end{example}
In general, we can conjecture that the sequences with generating function
$$\frac{1}{1-x-\alpha x^2}c\left(\frac{x^3}{(1-x-\alpha x^2)^2}\right)$$
have Hankel transforms that are $(1,-\alpha)$ Somos $4$ sequences. Such Hankel transforms begin
$$1, \alpha, -1 + \alpha, -1 + \alpha - \alpha^3, -2 \alpha + 3 \alpha^2 - \alpha^3 - \alpha^4, 1 - 3 \alpha +   3 \alpha^2 - 2 \alpha^3 + \alpha^4 - \alpha^6, $$
$$\quad 1 - 3 \alpha + 3 \alpha^2 + 3 \alpha^3 - 9 \alpha^4 +   6 \alpha^5 + \alpha^6 - 2 \alpha^7,$$  $$3 \alpha - 12 \alpha^2 + 19 \alpha^3 - 11 \alpha^4 - 3 \alpha^5 + 5 \alpha^6 + 2 \alpha^7 -   3 \alpha^8 + \alpha^{10},\ldots.$$
For instance, the Hankel transform of the sequence with generating function
$$\frac{1}{1-x+x^2}c\left(\frac{x^3}{(1-x+x^2)^2}\right)$$ is the $(1,1)$ Somos $4$ sequence \seqnum{A178627}$(n+1)$. The sequence \seqnum{A178627} is defined by the elliptic curve
$$E: y^2 + xy - y = x^3 - x^2 + x.$$
These examples show that such sequences merit more study. Further evidence of this is given by the following conjecture concerning their Hankel transforms.
\begin{conjecture} The Hankel transform of the sequence with generating function
$$\frac{1-x-x^2- \alpha x^3}{1-x-x^2-x^3}c\left(\frac{x^3(1-x-x^2-\alpha x^3)}{(1-x-x^2-x^3)^2}\right)$$ is given by
$$h_n= A_n(\alpha)(2-\alpha)^{B_n}$$
where
$$A_n(\alpha)=[x^n]\frac{1+(\alpha-2)x^2-(\alpha-2)x^3+(4\alpha-5)x^5-(\alpha-1)(\alpha-2)x^8}{1+3x^3-(\alpha-2)x^6},$$
and
$$B_n=[x^n] \frac{x(1-x+2x^2-2x^3+3x^4-3x^5+x^6)}{(1-x)^2(1-x^3)}.$$
\end{conjecture}
The sequence $B_n$ begins
$$0, 1, 1, 3, 4, 7, 9, 11, 15, 18, 21, 26, 30,\ldots.$$
The quasi-polynomial sequence \seqnum{A236337} is related to this.
\begin{example} For $\alpha=2$, we find that
$$\frac{1-x-x^2-2x^3}{1-x-x^2-x^3}c\left(\frac{x^3(1-x-x^2-2x^3)}{(1-x-x^2-x^3)^2}\right)=1.$$
Thus the Hankel transform is the sequence $1,0,0,0,\ldots$.
\end{example}
\begin{example} When $\alpha=-1$, we obtain the sequence that begins
$$1, 0, 0, 3, 3, 6, 18, 33, 69, 165, 351, 768, 1758, 3921, 8811, 20130,\ldots.$$
The Hankel transform of this begins
$$1, 0, -9, 0, 0, 0, -59049, 0, 43046721, 3486784401, 0,\ldots.$$
This is given by
$$h_n=3^{B_n}[x^n]\frac{1 - 3x^2 + 3x^3 - 9x^5 - 6x^8}{1 + 3x^3 + 3x^6}.$$
\end{example}
\section{Conversion of parameters}
Given a generating function of the form
$$g(x)=\left(\frac{1+ax+bx^2+cx^3}{1+dx+ex^2+fx^3}, \frac{mx^3(1+ax+bx^2+cx^3)}{(1+dx+ex^2+fx^3)^2}\right)\cdot c(x),$$ what are the corresponding parameters $p,q,s,u,v,w,t$?
We find the following.
\begin{align*}
p&=a-d\\
q&=-ad+b+d^2-e\\
s&=a(d^2-e)+c-bd-d^3+2de-f+m\\
u&=-d\\
v&=-e\\
w&=-f+2m\\
t&=m. \end{align*}
\begin{example} The doubly aerated large Schroeder numbers, with generating function
$$\left(\frac{1}{1-x^3},\frac{x^2}{(1-x^3)^2}\right)\cdot c(x),$$ correspond to the recurrence with parameters
$(0,0,2,0,0,3,1)$. This is the recurrence
$$a_n=3 a_{n-3}+\sum_{k=1}^{n-4} a_k a_{n-k-3},$$ with $a_0=1$, $a_1=0$, $a_2=0$ and $a_3=2$.
The resulting sequence
$$1, 0, 0, 2, 0, 0, 6, 0, 0, 22, 0, 0, 90, 0, 0, 394, 0, 0,\ldots$$ has a Hankel transform which begins
$$1, 0, -4, -8, 0, 128, 512, 0, -32768, -262144, 0,\ldots.$$
Dividing this by $2^n$ gives us the sequence that begins
$$1, 0, -1, -1, 0, 4, 8, 0, -128, -512, 0,\ldots.$$
This suggests that the generating function of this Hankel transform may satisfy the functional equation
$$1-x^2-x^3f(x)-f(x/2)=0.$$
\end{example}

\section{From elliptic curve to recurrences and Somos sequences}
We consider a particular example which illustrates a process that starts with an elliptic curve. Thus we consider the elliptic curve
$$E: y^2-3xy-y=x^3-x.$$
We start by solving this for $y$.
We get two solutions,
$$y^{-} = \frac{1+3x-\sqrt{1+2x+9x^2+4x^3}}{2},$$ and
$$y^{+} = \frac{1+3x+\sqrt{1+2x+9x^2+4x^3}}{2}.$$
Expanding the right hand side of the first solution $y^{-}$, we get
$$0,1,-2, 1, 3, -7, -4, 38, -27, -175, 384\ldots.$$
Expanding the right hand side of the second solution $y^{+}$, we get
$$1, 2, 2, -1, -3, 7, 4, -38, 27, 175, -384\ldots.$$
The solution $y^{-}$ has a generating function that can be expressed as
$$\frac{x(1-x^2)}{1+3x}c\left(\frac{x(1-x)^2}{(1+3x)^2}\right).$$ Thus the sequence that begins
$$1,-2, 1, 3, -7, -4, 38, -27, -175, 384\ldots$$ has $n$-th term given by
$$\sum_{k=0}^n \left(\sum_{j=0}^{k+1} \binom{k+1}{j}\binom{n+k-2j}{n-k-2j}(-1)^j(-3)^{n-k-2j}\right)C_k.$$
We are interested in the terms (apart from sign) that are common to these two solutions. We thus truncate the sequence to start with $2, -1, -3, 7, 4, -38, 27, 175, -384\ldots$.
This leads to the new generating function
$$g(x)=\left(\frac{1+3x+\sqrt{1+2x+9x^2+4x^3}}{2}-1-2x\right)/x^2=\frac{\sqrt{1+2x+9x^2+4x^3}-x-1}{2x^2}.$$
This sequence has general term
$$\sum_{k=0}^n \sum_{j=0}^{k+1} \binom{k+1}{j}\binom{n-j}{n-2k-j}(-1)^{n-k-j}2^{k+1-j} C_k,$$ or equivalently,
$$\sum_{k=0}^{n+1} \left(\sum_{j=0}^{k+1} \binom{k+1}{j}\binom{n+1+k-2j}{n+1-k-2j}(-1)^{j+1}(-3)^{n+1-k-2j}\right)C_k.$$
Now we find that
$$g(x)=\frac{2+x}{1+x}c\left(\frac{-x^2(2+x)}{(1+x)^2}\right)=\frac{\sqrt{1+2x+9x^2+4x^3}-x-1}{2x^2}.$$
Re-writing this as
$$2\frac{1+\frac{x}{2}}{1+x}c\left(\frac{-2x^2(1+\frac{x}{2})}{(1+x)^2}\right),$$ we see that the sequence
$2, -1, -3, 7, 4, -38, 27, 175, -384\ldots$ defined by the elliptic curve is $2d_n$ where $d_n$ is the solution of the recurrence $$d_n=-d_{n-1}-4d_{n-2}-2 \sum_{k=1}^{n-3} d_k d_{n-k-2},$$ with
$$d_0=1, d_1=-\frac{1}{2}, d_2=-\frac{3}{2}.$$
The Hankel transform of this sequence $2 d_n$ begins
$$2, -7, -57, 670, 23647, -833503, -147165662,\ldots$$ which is a $(1,16)$ Somos $4$ sequence.

We now form the generating function
$$\frac{1}{1-x+x^2 g(x)}=\frac{2}{1-3x+\sqrt{1+2x+9x^2+4x^3}}.$$ This gives us a sequence whose initial term is $1$ and whose Hankel transform prepends a $1$ to the previous transform.
Multiplying this by $x$ and reverting, and then dividing by $x$, we obtain the generating function
$$\frac{1+3x-\sqrt{1+6x+9x^2-4x^3-8x^4}}{2x^3}=\frac{1+2x}{1+3x}c\left(\frac{x^2(1+2x)}{(1+3x)^2}\right).$$
This expands to give the sequence $a_n$ that begins
$$1, -1, 3, -8, 22, -59, 155, -396, 978, -2310, 5122,\ldots,$$ which is thus the solution to the recurrence
$$a_n=-3 a_{n-1}+2 a_{n-3}+\sum_{k=1}^{n-4} a_k a_{n-k-3},$$ with
$a_0=1, a_1=-1, a_2=3, a_3=-8$.
The Hankel transform of this sequence begins
$$ 1, 2, 1, -7, -16, -57, -113, 670, 3983, 23647, 140576,\ldots.$$
This is a $(1,-2)$ Somos $4$ sequence. It is \seqnum{A178622}$(n+2)$. We note that the previous Hankel transform $2, -7, -57, 670, 23647, -833503, -147165662,\ldots$ is a bisection of this latter sequence.

Taking the second binomial transform $\sum_{k=0}^n \binom{n}{k}2^{n-k}a_k$ of the sequence $a_n$, we obtain the sequence that begins
$$1, 1, 3, 6, 14, 33, 79, 194, 482, 1214, 3090, 7936, \ldots.$$
We have met this already. It is $a_n(1,3,6,1,2,2,1)$ with generating function
$$\frac{1}{1-x-2x^2}c\left(\frac{x^3}{1-x-2x^2}\right).$$
It is interesting to note that the sequence $b_n$ that begins
$$0,2,1, 1, 3, 6, 14, 33, 79, 194, 482, 1214, 3090, 7936, \ldots$$ satisfies the recurrence
$$b_n=b_{n-1}+\sum_{i=0}^{n-3}b_i b_{n-i-1},$$ with $b_0=0, b_1=2, b_2=1$. See also \seqnum{A025243}.

We have the following coordinates for the points $nP(0,0)$ on the elliptic curve
$$E: y^2-3xy-y=x^3-x.$$

\begin{center}\begin{tabular}{|c|c|c|c|c|c|c|c|}
\hline
$x(nP)$ & $0$ & $-2$ & $-\frac{1}{4}$ & $14$ & $\frac{16}{49}$ & $\frac{-399}{256}$ & $\frac{-1808}{3249}$ \\
\hline
$y(nP)$ & $0$ & $-3$ & $\frac{5}{8}$ & $78$ & $\frac{55}{343}$ & $\frac{-11921}{4096}$ & $\frac{68464}{185193}$ \\
\hline
$\frac{y}{x}$ & $1$ & $\frac{3}{2}$ & $-\frac{5}{2}$ & $\frac{39}{7}$ & $\frac{55}{122}$ & $\frac{703}{912}$ & $-\frac{4279}{6441}$ \\
\hline
\end{tabular}\end{center}

We form the continued fraction
$$\cfrac{1}{1+x-
\cfrac{2x^2}{1+\frac{3x}{2}-
\cfrac{\frac{x^2}{4}}{1-\frac{5x}{2}+
\cfrac{14x^2}{1+\frac{39x}{7}+
\cfrac{\frac{16x^2}{49}}{1+\frac{55x}{112}-
\cfrac{\frac{399x^2}{256}}{1-\ldots}}}}}}.$$

This expands to give the  sequence $\tilde{a}_n$
$$1, -1, 3, -8, 22, -59, 155, -396, 978, -2310, 5122,\ldots$$ with g.f
$$\left(\frac{1+2x}{1+3x}, \frac{x^3(1+2x)}{(1+3x)^2} \right)\cdot c(x).$$
We have
$$\tilde{a}_n=\sum_{k=0}^n \sum_{j=0}^{k+1} \binom{k+1}{j}\binom{n-k-j}{n-3k-j} 2^j (-3)^{n-3k-j}C_k.$$

In this case, we have $\tilde{a}_n = a_n$, where $a_n$ is the sequence obtained starting from solving the elliptic curve equation.

We briefly look at the shifted sequence
$$ -1, -3, 7, 4, -38, 27, 175, -384\ldots$$ of the sequence first encountered in this section.
This has a generating function given by
$$-\left(\frac{1+4x}{1+x+4x^2}, \frac{x^3(1+4x)}{(1+x+4x^2)^2}\right)\cdot C(x).$$ It has a Hankel transform that begins
$$ -1, -16, 113, 3983, -140576, -14871471, \ldots.$$
Again, this is a $(1, 16)$ Somos $4$ sequence.

The elliptic curve
$$E: y^2-3xy-y=x^3-x$$ thus gives rise to the following Riordan arrays
\begin{itemize}
\item $\left(\frac{1-x^2}{1+3x}, \frac{x(1-x^2)}{(1+3x)^2}\right)$\\
\item $\left(\frac{2+x}{1+x}, -\frac{x^2(2+x)}{(1+x)^2}\right)$\\
\item $\left(\frac{1+4x}{1+x+4x^2}, \frac{x^3(1+4x)}{(1+x+4x^2)^2}\right)$\\
\item $\left(\frac{1}{1-3x}, \frac{x(2+x^2)}{(1-3x)^2}\right)$\\
\item $\left(\frac{1+2x}{1+3x},\frac{x^3(1+2x)}{(1+3x)^2}\right)$\\
\item $\left(\frac{1}{1-x-2x^2}, \frac{x^3}{(1-x-2x^2)^2}\right)$
\end{itemize}
with their corresponding recurrences.
\section{The family $E_t: y^2+4xy+y=x^3+(t-1)x^2+tx$}
The family of elliptic curves
$$E_t: y^2+4xy+y=x^3+(t-1)x^2+tx$$ has the property that each of its curves passes through the points
$(0,0), (-1,1)$ and $(-1,2)$.
Solving the equation
$$y^2+4xy+y=x^3+(t-1)x^2+tx$$ gives
$$y=-\frac{1+4x\pm \sqrt{1+4(t+2)x+4(t+3)x^2+4x^3}}{2}.$$
For instance, $-\frac{1+4x- \sqrt{1+4(t+2)x+4(t+3)x^2+4x^3}}{2}$ expands to give
$$0, t, -(t^2+3t+1), 2t^3+10t^2+14t+5,\ldots.$$
The generating function of this sequence can be expressed as
$$\frac{x(t+x(t-1)+x^2}{1+4x}c\left(\frac{-x(t+x(t-1)+x^2)}{(1+4x)^2}\right).$$
As before, we are interested in the terms (up to sign) that are common to the two solutions. Thus we focus on the  sequence $b_n$ that begins
$$(t^2+3t+1), -(2t^3+10t^2+14t+5),\ldots.$$
This has generating function
$$g(x)=\frac{1+3t+t^2-x}{1+2(t+2)x}c\left(\frac{x^2(1+3t+t^2-x)}{(1+2(t+2)x)^2}\right).$$
We have the following conjecture.
\begin{conjecture}
The Hankel transform of $b_n$ is an
$$((2t^3+10t^2+14t+5)^2, -3t^8-40t^7-222t^6-666t^5-$$
$$\quad\quad\quad 1173t^4-1230t^3-740t^2-29(8t+1))$$ Somos $4$ sequence.
\end{conjecture}
The generating function $\frac{1}{1-x-x^2g(x)}$ will now give us a sequence whose initial term is $1$, which has the same Hankel transform, with a prepended $1$. We obtain that
$$\frac{1}{1-x-x^2g(x)}=\frac{1}{1-2(t+3)x}c\left(\frac{-x(x^2-(t+2)(t+3)x+2t+5)}{(1-2(t+3)x)^2}\right).$$
We now revert the generating function $\frac{x}{1-x-x^2g(x)}$, and divide the result by $x$, to get the generating function
$$a(x)=\frac{1+(2t+5)x}{1+2(t+3)x+(t+2)(t+3)x^2}c\left(\frac{x^3(1+(2t+5)x)}{(1+2(t+3)x+(t+2)(t+3)x^2)^2}\right).$$
This expands to give a sequence $a_n$ that begins
$$1, -1, - t(t + 3), 2t^3 + 13t^2 + 23t + 7, - 3t^4 - 30t^3 - 103t^2 - 134t - 44,\ldots.$$
The sequence $a_n$ therefore satisfies the following recurrence.
$$a_n = - 2·(t + 3) a_{n-1}- (t + 2)·(t + 3)a_{n-2}+2 a_{n-3}+ \sum_{k=1}^{n-k-4} a_k a_{n-k-3},$$
with $a_0=1, a_1=-1, a_2=- t(t + 3), a_3=2t^3 + 13t^2 + 23t + 7$.
We have the following conjecture concerning the Hankel transform of this sequence.
\begin{conjecture} The Hankel transform of the sequence whose generating function is given by
$$\frac{1+(2t+5)x}{1+2(t+3)x+(t+2)(t+3)x^2}c\left(\frac{x^3(1+(2t+5)x)}{(1+2(t+3)x+(t+2)(t+3)x^2)^2}\right)$$
is a $(1, t^2+3t+1)$ Somos $4$ sequence.
\end{conjecture}
\begin{example} When $t=0$, we get the sequence with generating function
$$\frac{1+5x}{1+6x+6x^2}c\left(\frac{x^3(1+5x)}{(1+6x+6x^2)^2}\right),$$ which begins
$$1, -1, 0, 7, -44, 223, -1060, 4920, -22626, 103719, -475214,\ldots.$$
This sequence satisfies then the recurrence
$$a_n=-6a_{n-1}-6 a_{n-3}+2 a_{n-3}+\sum_{k=1}^{n-4} a_k a_{n-k-3}$$ with
$$a_0=1, a_1=-1, a_2=0, a_3=7.$$
The Hankel transform of this sequence is then the $(1,1)$ Somos $4$ sequence \seqnum{A157101} which begins
$$1, -1, -5, -4, 29, 129, -65, -3689, -16264,\ldots.$$
The associated elliptic curve is
$$E_0 : y^2+4xy+y=x^3-x^2.$$
\end{example}
\begin{example} When $t=-3$, we obtain the sequence $a_n$ with generating function
$$ (1-x)c(x^3(1-x))=\frac{1-\sqrt{1-4x^3+4x^4}}{2x^3}$$ which begins
$$1, -1, 0, 1, -2, 1, 2, -6, 6, 3, -20,\ldots.$$
This sequence then satisfies  the recurrence
$$a_n=2 a_{n-3}+\sum_{k=1}^{n-4} a_k a_{n-k-3}$$ with
$$a_0=1, a_1=-1, a_2=0, a_3=1.$$
Its Hankel transform is the $(1,1)$ Somos $4$ sequence \seqnum{A006769}$(n+2)$ that begins
$$1, -1, 1, 2, -1, -3, -5, 7, -4, -23, 29,\ldots.$$
The associated elliptic curve is
$$E_{(-3)}: y^2+4xy+y=x^3-4x^2-3x.$$
We can relate the sequence to the coordinates of the multiples of the point $(0,0)$ on this curve in the following way. Taking the second binomial transform of $(-1)^n a_n$, given by
$$d_n=\sum_{k=0}^n 2^{n-k}(-1)^k a_k,$$ and then taking the INVERT$(4)$ transform of $d_n$, we arrive at the sequence $\tilde{a}_n$ that begins
$$1, -1, 0, -1, -2, -5, -10, -14, 6, 145, 720, 2618, 7850, 19389, 35016, \ldots.$$
Due to the invariance of the Hankel transform under binomial and INVERT transforms, this sequence has the same Somos $4$ Hankel transform as $a_n$.
The sequence $\tilde{a}_n$ has a generating function that can be expressed as
$$\frac{1-x}{1+4x-4x^2}c\left(\frac{x(4-4x-x^2-x^3)}{(1+4x-4x^2)^2}\right).$$
We can express this as the continued fraction
$$\cfrac{1}{1+x+
\cfrac{x^2}{1-2x-
\cfrac{x^2}{1-x+
\cfrac{2x^2}{1-\frac{7x}{2}+
\cfrac{\frac{x^2}{4}}{1-\frac{9x}{2}+
\cfrac{6x^2}{1-\ldots}}}}}}.$$
This corresponds to the following $x$ and $y$ coordinates of the multiples of $P(0,0)$ on $E_{(-3)}$.
\begin{center}\begin{tabular}{|c|c|c|c|c|c|c|c|}
\hline
$x(nP)$ & $0$ & $1$ & $-1$ & $2$ & $\frac{1}{4}$ & $6$ & $\frac{-5}{9}$ \\
\hline
$y(nP)$ & $0$ & $-2$ & $1$ & $-7$ & $\frac{-9}{8}$ & $2$ & $\frac{38}{27}$ \\
\hline
$\frac{y}{x}$ & $1$ & $\-2$ & $-1$ & $\frac{-7}{2}$ & $\frac{-9}{2}$ & $\frac{1}{3}$ & $-\frac{38}{15}$ \\
\hline
\end{tabular}\end{center}
Thus we can write the generating function $\sum_{n=0}^{\infty} \tilde{a}_n t^n$ as
$$\cfrac{1}{1+z(0)t+
\cfrac{x(1)t^2}{1+z(1)t+
\cfrac{x(2)t^2}{1+z(2)t+\cdots}}},$$ where
$z(n)=\frac{y(nP)}{x(nP)}$ and $x(n)=x(nP)$.
The sequence $\tilde{a}_n$ is the moment sequence for the family of orthogonal polynomials $Q_n(x)$ with a coefficient matrix that begins
$$\left(
\begin{array}{ccccccc}
 1 & 0 & 0 & 0 & 0 & 0 & 0 \\
 1 & 1 & 0 & 0 & 0 & 0 & 0 \\
 -1 & -1 & 1 & 0 & 0 & 0 & 0 \\
 0 & -1 & -2 & 1 & 0 & 0 & 0 \\
 -2 & 3/2 & 8 & -11/2 & 1 & 0 & 0 \\
 9 & -9 & -35 & 35 & -10 & 1 & 0 \\
 -9 & 15 & 82/3 & -57 & 107/3 & -29/3 & 1 \\
\end{array}
\right).$$
The inverse $M$ of this matrix (the moment matrix) contains $\tilde{a}_n$ as its first column. The matrix $M$ begins
$$\left(
\begin{array}{ccccccc}
 1 & 0 & 0 & 0 & 0 & 0 & 0 \\
 -1 & 1 & 0 & 0 & 0 & 0 & 0 \\
 0 & 1 & 1 & 0 & 0 & 0 & 0 \\
 -1 & 3 & 2 & 1 & 0 & 0 & 0 \\
 -2 & 7 & 3 & 11/2 & 1 & 0 & 0 \\
 -5 & 15 & -1 & 22 & 10 & 1 & 0 \\
 -10 & 24 & -30 & 147/2 & 61 & 29/3 & 1 \\
\end{array}
\right).$$
The production matrix of $M$ then begins
$$\left(
\begin{array}{ccccccc}
 -1 & 1 & 0 & 0 & 0 & 0 & 0 \\
 -1 & 2 & 1 & 0 & 0 & 0 & 0 \\
 0 & 1 & 1 & 1 & 0 & 0 & 0 \\
 0 & 0 & -2 & 7/2 & 1 & 0 & 0 \\
 0 & 0 & 0 & -1/4 & 9/2 & 1 & 0 \\
 0 & 0 & 0 & 0 & -6 & -1/3 & 1 \\
 0 & 0 & 0 & 0 & 0 & 5/9 & 38/15 \\
\end{array}
\right).$$
The tri-diagonal nature of this production matrix shows that we are in the presence of a family of orthogonal polynomials. The diagonal and sub-diagonal contain the $x$ coordinates and the $y/x$-ratios of the multiples of the point $P(0,0)$ on $E_{(-3)}$.

The family $Q_n(t)$  is defined by the three term recurrence
$$Q_n(t)=\left(t+\frac{x((n-1)P)}{y((n-1)P)}\right)Q_{n-1}(t)+x((n-1)P)Q_{n-2}(t),$$ where $P=(0,0)$ on $E_{(-3)}$ and we have $Q_0(t)=1, Q_1(t)=t+1$.

Using results about the divisibility polynomials $\psi_n$ of an elliptic curve \cite{Swart}, we can express the $x$ and $y$ coordinates of multiples of the point $P(0,0)$ for the curve $E_{(-3)}$ as follows.
$$x(nP)=\frac{-\psi_{n-1} \psi_{n+1}}{\psi_n^2},$$ and
$$y(nP)=\frac{1}{2}\left(\frac{\psi_{2n}}{\psi_n^4}-\left(1-4 \frac{\psi_{n-1} \psi_{n+1}}{\psi_n^2}\right)\right),$$
where we have used the shorthand $\psi_n=\psi_n(0,0)$ in the above expressions. Alternatively, we have
$$x((n+1)P)=\frac{-\tilde{h}_{n-1}\tilde{h}_{n+1}}{\tilde{h}_n^2},$$ and
$$y((n+1)P)=\frac{1}{2}\left(\frac{\tilde{h}_{2(n+1)}}{\tilde{h}_n^4}+a_1 \frac{\tilde{h}_{n-1}\tilde{h}_{n+1}}{\tilde{h}_n^2}-a_3\right),$$ where
$a_1=4$ and $a_3=1$. Here, $\tilde{h}_n$ is the Hankel transform of $\tilde{a}_n$.

We note finally that an application of the Stieltjes-Perron transform suggests that the absolutely continuous part of the associated measure is given by
$$\frac{1}{\pi}\frac{\sqrt{-x^4-4(2x^3+6x^2+7x+3)}}{2(4x^2-1)}.$$
\end{example}

\section{The case of $E : y^2+ a x y+ y=x^3+ b x^2 + cx$}
We let $P=P(0,0)$ be the point $(0,0)$ on the elliptic curve
$$E : y^2+ a x y+ y=x^3+ b x^2 + cx,$$ and let $\tilde{a}_n$ be the sequence with generating function
$$\tilde{a}(t)=\cfrac{1}{1+t+
\cfrac{x_1t^2}{1+\frac{y_1}{x_1}t+
\cfrac{x_2t^2}{1+\frac{y_2}{x_2}t+
\cfrac{x_3t^2}{1+\frac{y_3}{x_3}t+\cdots}}}},$$
where $x_n=x(nP)$ and $y_n=y(nP)$ are the coordinates of the multiples $nP$ of the point $P$ on the curve.
We can conjecture that the generating function of $\tilde{a}_n$ is given by
$$(g(x), f(x))\cdot c(x),$$
where
$$g(x)=\frac{1}{1-(a+2 \gamma)x-(b-2 \gamma(a+c))x^2},$$
and
\begin{scriptsize}
$$f(x)=\frac{x(x^3(a + c)(\gamma^2(a + c) - b\gamma + 1) - x^2(2c\gamma^2 + \gamma(a^2 + a(3·c - 2) - b) + 1) + x \gamma(2a + 2c - 1) - \gamma)}{(1-(a+2 \gamma)x-(b-2 \gamma(a+c))x^2)^2}.$$
\end{scriptsize}
Here, we have used the notation
$$\gamma = c-1.$$
We see that there are simplifications when $c=1$. In this case, we find that the generating function of $\tilde{a}_n$ is given by
$$\frac{1-(a+1)x}{1-ax-bx^2}c\left(\frac{-x^3(1-(a+1)x)}{(1-ax-bx^2)^2}\right).$$
This allows us to give a closed form expression for $\tilde{a}_n$ when $c=1$.
\begin{scriptsize}
$$\tilde{a}_n=\sum_{k=0}^n (\sum_{j=0}^{k+1} \binom{k+1}{j}(-a-1)^j \sum_{i=0}^{n-3k-j} \binom{2k+i}{i}\binom{i}{n-3k-i-j}b^{n-3k-i-j}a^{2i+3k+j-n})(-1)^k C_k.$$
\end{scriptsize}
This sequence begins
$$1, -1, b - a, - a^2 + ab - b - 1, - a^3 + a^2 b - a(2b + 1) + b^2 + 2,$$
$$\quad - a^4 + a^3b - a^2(3b + 1) + a(2b^2 + 4) - b^2 - 3b - 1,\ldots,$$
with a Hankel transform that begins
$$1, -a + b - 1, - a^2 + a(b - 3) + 2b - 3, $$
$$\quad a^3 + a^2(2 - 3b) + a(3b^2 - 5b) - b^3 + 3b^2 - b - 2,\ldots,$$
which we conjecture to be a $(1, a-b+1)$ Somos $4$ sequence.

The sequence $\tilde{a}_n$ for $c=1$ satisfies the convolution recurrence
$$\tilde{a}_n= a \tilde{a}_{n-1}+ b \tilde{a}_{n-2}-2 \tilde{a}_{n-3} - \sum_{k=1}^{n-k-4} \tilde{a}_k \tilde{a}_{n-k-3},$$ with
$\tilde{a}_0=1, \tilde{a}_1=-1, \tilde{a}_2=b-a, \tilde{a}_3=-a^2+ab-b-1$.

For the related sequence $a_n$, calculated from the solution of the elliptic curve equation, we can calculate that its generating function is given by
$$(g(x), f(x)) \cdot c(x)$$ where
$$g(x)=\frac{1+(1+a+2c)x}{1+(a+2c+2)x+(a(c+1)+(c+1)^2-b)x^2},$$
$$f(x)=\frac{x^3(1+(1+a+2c)x)}{(1+(a+2c+2)x+(a(c+1)+(c+1)^2-b)x^2)^2}.$$
This means that $a_n$ satisfies the following convolution recurrence.
\begin{scriptsize}
$$a_n=-(a+2(c+1))a_{n-1}+(b-a(c+1)-(c+1)^2)a_{n-2}+2 a_{n-3}+\sum_{k=1}^{n-4} a_k a_{n-k-3},$$
\end{scriptsize}
with $a_0=1, a_1=-1, a_2=1+b-ac-c^2, a_3=(3+a+2c)(-b+c(a+c))$.

The relationship between the generating function $\tilde{a}(x)=\sum_{n=0}^{\infty} \tilde{a}_n x^n$ and the generating function $a(x)=\sum_{n=0}^{\infty} a_n x^n$ is the following: $\tilde{a}(x)$ is the INVERT$(-(c+1))$ of the $(c+1)$-st binomial transform of $a(-x)$.

\section{Conclusions}
We have found that the language of Riordan arrays is an appropriate one to investigate structural aspects of the solutions of generalized Catalan or Schroeder recurrences. The appearance of the Catalan numbers is explained by the quadratic nature of the convolution recurrences. This quadratic theme is continued in terms of solving elliptic curve equations in the quadratic term, which leads to generating functions that are Riordan array solutions of convolution equations. Often, the Hankel transform of the solution sequences give rise to Somos $4$ sequences. By using the $x$ and $y$ coordinates of multiples of a special point on the elliptic curves, we can obtain Stieltjes continued fraction expressions for the generating functions of related integer sequences. In specific instances, this leads to orthogonal polynomials and their three term recurrences, opening the door to the investigation of associated measures. Note that we work with points on curves for which the division polynomials never evaluate to $0$.

The association of Somos $4$ sequences and the Hankel transform is implicit in previous work (see the Appendix) \cite{Hone, Poorten, Swart}. What is apparently new is the use of the $\frac{y(nP)}{x(nP)}$ ratios to define \emph{integer} sequences whose Hankel transforms furnish Somos sequences.

Proofs of results concerning Hankel transforms and Somos sequences can be elusive (but see \cite{Hone, KrattL, Kratt, Xin, Yura}, hence we couch some proposed results as conjectures. It is hoped that further insight will remedy this in the future.

\section{Appendix: Orthogonal polynomials, Riordan arrays and the Hankel transform}
By an
\emph{orthogonal polynomial sequence} $(p_n(x))_{n \ge 0}$ we shall
understand \cite{Chihara, Gautschi} an infinite sequence of
polynomials
$p_n(x)$, $n\ge 0$, of degree $n$, with real coefficients (often integer
coefficients) that are mutually orthogonal on an interval
$[x_0,x_1]$ (where
$x_0=-\infty$ is allowed, as well as $x_1=\infty$), with
respect to a weight function $w:[x_0,x_1] \to \mathbb{R}$ \,:
$$\int_{x_0}^{x_1}
p_n(x)p_m(x)w(x)dx = \delta_{nm}\sqrt{h_nh_m},$$ where
$$\int_{x_0}^{x_1} p_n^2(x)w(x)dx=h_n.$$ We assume that $w$ is strictly positive on the interval $(x_0,x_1)$.  Every such
sequence
obeys a so-called
``three-term recurrence" \,: $$
p_{n+1}(x)=(a_nx+b_n)p_n(x)-c_np_{n-1}(x)$$ for coefficients
$a_n$, $b_n$ and $c_n$ that depend on $n$ but
not $x$. We note that if $$p_j(x)=k_jx^j+k'_jx^{j-1}+\ldots
\qquad j=0,1,\ldots$$ then $$a_n=\frac{k_{n+1}}{k_n},\qquad
b_n=a_n\left(\frac{k'_{n+1}}{k_{n+1}}-\frac{k'_n}{k_n}\right),
\qquad c_n=a_n\left(\frac{k_{n-1}h_n}{k_nh_{n-1}}\right),$$
where
$$h_i=\int_{x_0}^{x_1} p_i(x)^2 w(x)\,dx.$$
Since the degree
of $p_n(x)$ is $n$, the coefficient array of the polynomials
is
a lower triangular (infinite) matrix. In the case of monic
orthogonal
polynomials the diagonal elements of this array will all be
$1$. In this case, we can write the three-term recurrence as
$$p_{n+1}(x)=(x-\alpha_n)p_n(x)-\beta_n p_{n-1}(x), \qquad
p_0(x)=1,\qquad p_1(x)=x-\alpha_0.$$

\noindent The \emph{moments} associated to the orthogonal
polynomial sequence are the numbers $$\mu_n=\int_{x_0}^{x_1}
x^n w(x)dx.$$

\noindent We can find $p_n(x)$, $\alpha_n$ and $\beta_n$ from a
knowledge
of these moments. To do this, we let $\Delta_n$ be the Hankel
determinant
$|\mu_{i+j}|_{i,j\ge 0}^n$ and $\Delta_{n,x}$ be the same
determinant, but with the last row equal to $1,x,x^2,\ldots$.
Then
$$p_n(x)=\frac{\Delta_{n,x}}{\Delta_{n-1}}.$$ More generally,
we let $H\left(\begin{array}{ccc} u_1&\ldots& u_k\\
v_1 & \ldots & v_k\end{array}\right)$ be the determinant
of
Hankel type with $(i,j)$-th term $\mu_{u_i+v_j}$. Let
$$\Delta_n=H\left(\begin{array}{cccc} 0&1&\ldots&n\\
0&1&\ldots&n\end{array}\right),\qquad
\Delta'=H\left(\begin{array}{ccccc}
0&1&\ldots&n-1&n\\ 0&1&\ldots&n-1&n+1\end{array}\right).$$
Then
we have
$$\alpha_n=\frac{\Delta'_n}{\Delta_n}-\frac{\Delta'_{n-1}}{\Delta_{n-1}},\qquad
\beta_n=\frac{\Delta_{n-2} \Delta_n}{\Delta_{n-1}^2}.$$

Of importance to this study are the following results (the first is the well-known ``Favard's Theorem''), which we essentially reproduce from
\cite{Kratt}.
\begin{theorem} \cite{Kratt} (Cf. \cite{Viennot}, Th\'eor\`eme $9$ on p.I-4, or \cite{Wall}, Theorem $50.1$). Let $(p_n(x))_{n\ge 0}$ be a sequence of monic polynomials, the polynomial $p_n(x)$ having degree $n=0,1,\ldots$ Then the sequence $(p_n(x))$ is (formally) orthogonal if and only if there exist sequences $(\alpha_n)_{n\ge 0}$ and $(\beta_n)_{n\ge 1}$ with $\beta_n \neq 0$ for all $n\ge 1$, such that the three-term recurrence
$$p_{n+1}=(x-\alpha_n)p_n(x)-\beta_n (x), \quad \text{for}\quad n\ge 1, $$
holds, with initial conditions $p_0(x)=1$ and $p_1(x)=x-\alpha_0$.
\end{theorem}

\begin{theorem} \cite{Kratt} (Cf. \cite{Viennot}, Proposition 1, (7), on p. V-$5$, or \cite{Wall}, Theorem $51.1$). Let $(p_n(x))_{n\ge 0}$ be a sequence of monic polynomials, which is orthogonal with respect to some functional L. Let
$$p_{n+1}=(x-\alpha_n)p_n(x)-\beta_n (x), \quad \text{for}\quad n\ge 1, $$ be the corresponding three-term recurrence which is guaranteed by Favard's theorem. Then the generating function
$$g(x)=\sum_{k=0}^{\infty} \mu_k x^k $$ for the moments $\mu_k=L(x^k)$ satisfies
$$g(x)=\cfrac{\mu_0}{1-\alpha_0 x-
\cfrac{\beta_1 x^2}{1-\alpha_1 x -
\cfrac{\beta_2 x^2}{1-\alpha_2 x -
\cfrac{\beta_3 x^2}{1-\alpha_3 x -\cdots}}}}.$$
\end{theorem}

Given a family of monic orthogonal polynomials
$$p_{n+1}(x)=(x-\alpha_n)p_n(x)-\beta_n p_{n-1}(x), \qquad
p_0(x)=1,\qquad p_1(x)=x-\alpha_0,$$
we can write $$p_n(x)=\sum_{k=0}^n a_{n,k}x^k.$$ Then we have
$$\sum_{k=0}^{n+1}a_{n+1,k}x^k=(x-\alpha_n)\sum_{k=0}^n
a_{n,k}x^k - \beta_n
\sum_{k=0}^{n-1}a_{n-1,k}x^k$$ from which we deduce
\begin{equation}\label{OP_1}a_{n+1,0}=-\alpha_n
a_{n,0}-\beta_n
a_{n-1,0}\end{equation}
and \begin{equation}\label{OP_2}a_{n+1,k}=a_{n,k-1}-\alpha_n
a_{n,k}-\beta_n a_{n-1,k}\end{equation}
The question immediately arises as to the conditions under which a Riordan array $(g,f)$ can be the coefficient array
 of a family of orthogonal polynomials.  A partial answer is given by the following proposition.

\begin{proposition} Every Riordan array of the form
$$\left(\frac{1}{1+rx+sx^2},\frac{x}{1+rx+sx^2}\right)$$ is the coefficient array of a family of monic orthogonal polynomials.
\end{proposition}

We note that in this case the three-term recurrence coefficients $\alpha_n$ and $\beta_n$ are constants. We can strengthen this result as follows.
\begin{proposition} Every Riordan array of the form
$$\left(\frac{1-\lambda x - \mu x^2}{1+rx+sx^2},\frac{x}{1+rx+sx^2}\right)$$ is the coefficient array of a family of monic orthogonal polynomials.
\end{proposition}

\begin{proposition} The elements in the left-most column of
$$L=\left(\frac{1-\lambda x - \mu x^2}{1+rx+sx^2},\frac{x}{1+rx+sx^2}\right)^{-1}$$ are the moments corresponding to the family of orthogonal polynomials with coefficient array $L^{-1}$.
\end{proposition}

We have in fact the following proposition, which characterizes those orthogonal polynomials that can be defined by Riordan arrays in terms of the Chebyshev polynomials of the second kind.
\begin{proposition} The Riordan array $\left(\frac{1}{1+rx+sx^2},\frac{x}{1+rx+sx^2}\right)$ is the coefficient array of the modified Chebyshev polynomials of the second kind given by
$$P_n(x)=(\sqrt{s})^n U_n\left(\frac{x-r}{2\sqrt{s}}\right), \quad n=0,1,2,\ldots$$
\end{proposition}

\begin{corollary}
The Riordan array $\left(\frac{1-\lambda x - \mu x^2}{1+rx+sx^2},\frac{x}{1+rx+sx^2}\right)$ is the coefficient array of the generalized Chebyshev polynomials of the second kind given by
$$Q_n(x)=(\sqrt{s})^n U_n\left(\frac{x-r}{2\sqrt{s}}\right)-\lambda(\sqrt{s})^{n-1} U_{n-1}\left(\frac{x-r}{2\sqrt{s}}\right)-\mu(\sqrt{s})^{n-2} U_{n-2}\left(\frac{x-r}{2\sqrt{s}}\right), \quad n=0,1,2,\ldots$$

\end{corollary}

The Hankel transform \cite{Layman} of a sequence $a_n$ is the sequence $h_n$  of determinants $| a_{i+j}|_{0 \le i,j \le n}$. For instance, the Hankel transform of the Catalan numbers is given by the all $1$'s sequence
$$1,1,1,1,1,\ldots.$$ The Hankel transform can have combinatorial significance; for instance, the Hankel transform of the ternary numbers begins
$$1, 2, 11, 170, 7429, 920460, 323801820, 323674802088, \ldots.$$
This sequence (\seqnum{A051255}) counts the number of cyclically symmetric transpose complement plane partitions in a $(2n+2) \times (2n+2) \times (2n+2)$ box.

If the sequence $a_n$ has a generating function $g(x)$, then the bivariate generating function of the Hankel matrix $|a_{i+j}|_{i,j \ge 0}$ is given by
$$\frac{x g(x)- y g(y)}{x-y}.$$

In the case that a sequence $a_n$ has g.f. $g(x)$ expressible in the continued fraction form \cite{Wall}
$$g(x)=\cfrac{a_0}{1-\alpha_0 x-
\cfrac{\beta_1 x^2}{1-\alpha_1 x-
\cfrac{\beta_2 x^2}{1-\alpha_2 x-
\cfrac{\beta_3 x^2}{1-\alpha_3 x-\cdots}}}}$$ then
we have the Heilermann formula \cite{Kratt}
\begin{equation}\label{Kratt} h_n = a_0^{n+1} \beta_1^n\beta_2^{n-1}\cdots \beta_{n-1}^2\beta_n=a_0^{n+1}\prod_{k=1}^n
\beta_k^{n+1-k}.\end{equation}
Note that this independent from $\alpha_n$.

We note that $\alpha_n$ and $\beta_n$ are in general not integers (even if both $a_n$ and $h_n$ are integer valued).  It is clear also that a Hankel transform has an infinite number of pre-images, since we are free to assign values to the $\alpha_n$ coefficients. For instance, a sequence $a_n$  and its binomial transform $\sum_{k=0}^n \binom{n}{k}$ have the same Hankel transform, and the expansion of the INVERT transform of $g(x)$ given by
$\frac{g(x}{1-xg(x)}$ and that of $g(x)$ will have the same Hankel transform.

 Somos $4$ sequences are most commonly associated with the
$x$-coordinate of rational points on an elliptic curve \cite{Hone, Swart, Poorten}. The link between these sequences and Hankel transforms is
made explicit in Theorem 7.1.1 of \cite{Swart}, for instance. Letting $nP$ denote the $n$-fold sum $P+ \cdots +P$ of points on an elliptic curve $E$, this
result implies the following: if $P=(\bar{x},\bar{y})$ and $Q=(x_0,y_0)$ are two distinct non-singular rational points on an elliptic curve
$E$, denote,
for all $n \in \mathbb{Z}$ such that $Q+nP\neq \mathcal{O}$ (the point at infinity on $E$), by $(x_n,y_n)$ the coordinates of the point
$Q+nP$.  Then under these circumstances the numbers determined by
$$s_n =(-1)^{\binom{n+1}{2}}(x_{n-1}-\bar{x})(x_{n-2}-\bar{x})^2\cdots(x_{1}-\bar{x})^{n-1}(x_0-\bar{x})^n s_0\left(\frac{s_0}{s_{-1}}\right)^n$$
are elements of a Somos 4 sequence (given appropriate $s_0,s_{-1}\neq 0$). We can re-write this as
$$s_n=s_0\left(\frac{s_0}{s_{-1}}\right)^n\prod_{k=0}^{n-1} (\bar{x}-x_k)^{n-k},$$ and we see that this is in the form of a Hankel transform.

\bigskip
\hrule
\bigskip
\noindent 2010 {\it Mathematics Subject Classification}: Primary 11B37;
 Secondary 11B83, 15B36,11B37, 05A15, 11Y55, 14H52, 42C05.
\noindent \emph{Keywords:}  Convolution recurrence, generating function, Catalan number, Schr\"oder numbers, Riordan array, Hankel transform, Somos sequence, elliptic curve, orthogonal polynomial.

\bigskip
\hrule
\bigskip
\noindent (Concerned with sequences
\seqnum{A000045},
\seqnum{A000108},
\seqnum{A001003},
\seqnum{A001519},
\seqnum{A004148},
\seqnum{A004149},
\seqnum{A006318},
\seqnum{A006769},
\seqnum{A014431},
\seqnum{A023431},
\seqnum{A025235},
\seqnum{A025243},
\seqnum{A025272},
\seqnum{A048990},
\seqnum{A084782},
\seqnum{A085139},
\seqnum{A086581},
\seqnum{A091561},
\seqnum{A128720},
\seqnum{A157101},
\seqnum{A162985},
\seqnum{A174171},
\seqnum{A174403},
\seqnum{A174404},
\seqnum{A178376},
\seqnum{A178384},
\seqnum{A178622},
\seqnum{A178627},
\seqnum{A187256}, and
\seqnum{A236337}.)


\begin{thebibliography}{99}

\bibitem{PS} P. Barry, Riordan Pseudo-Involutions, Continued Fractions and Somos $4$ Sequences, preprint ,2018. Available at \url{https://arxiv.org/abs/1807.05794}.

\bibitem{Book} P. Barry, \emph{Riordan Arrays: A Primer}, Logic Press, 2017.

\bibitem{classical} P. Barry and A. Mwafise, Classical and semi-classical orthogonal
polynomials defined by Riordan arrays, \emph{J. Integer Seq.}, \textbf{21}, (2018),
and their moment sequences,  \href{https://cs.uwaterloo.ca/journals/JIS/VOL21/Barry/barry362.html} {Article 18.1.5}.

\bibitem{Barry_moments} P. Barry, Riordan arrays, orthogonal polynomials as moments, and Hankel transforms, \emph{J. Integer Seq.}, \textbf{14} (2011), \href{https://cs.uwaterloo.ca/journals/JIS/VOL14/Barry1/barry97r2.html} {Article 11.2.2}.

\bibitem{CFT} P. Barry, Continued fractions and
transformations of integer sequences, \emph{J. Integer Seq.}, \textbf{12} (2009),
\href{https://cs.uwaterloo.ca/journals/JIS/VOL12/Barry3/barry93.html} {Article 09.7.6}.

\bibitem{Conj} X-K Chang and X-B Hu, A conjecture based on Somos-$4$ sequenc and its extension, {Linear Algebra Appl.}, \textbf{436} (2012) 4285--4295.

\bibitem{Chihara}
T. S. Chihara,  {\it An Introduction to Orthogonal Polynomials},
Dover Publicatons, 2011.

\bibitem{Gautschi}
W. Gautschi, {\it Orthogonal Polynomials: Computation and
Approximation}, Clarendon Press, 2004.

\bibitem{Henrici} P. Henrici, \emph{Applied and Computational Complex Analysis, vol. $3$}, John Wiley \& Sons, 1993.

\bibitem{Hone} A. Hone, Elliptic curves and quadratic recurrence sequences, \emph{Bull. Lond. Math. Soc.}, \textbf{37} (2005),  161–-171.

\bibitem{KrattL} C. Krattenthaler, Lattice path enumeration, in M. Bona (ed) \emph{Handbook of Enumerative Combinatorics}, Chapman \& Hall/CRC Press, available electronically at \url{https://arxiv.org/abs/1503.05930}.

\bibitem{Kratt}
C. Krattenthaler, Advanced determinant calculus: a
complement, \emph{Lin. Alg. Appl.}, {\bf 411} (2005),
68–-166.

\bibitem{Layman} J. W. Layman, The Hankel transform and some
    of
    its properties, \emph{J. of Integer Seq.}, {\bf
    4}  (2001), \href{http://www.cs.uwaterloo.ca/journals/JIS/VOL4/LAYMAN/hankel.html} {Article 01.1.5}.

\bibitem{MC} D. Merlini, R. Sprugnoli, and M.~C.
    Verri, The Method of Coefficients, \emph{Amer. Math. Monthly}, \textbf{114} (2007), 40--57.

\bibitem{Qi} F. Qi, B-N Guo, Some explicit and recursive formulas of the large and little
Schr\"oder numbers, \emph{Arab J Math Sci}, \text{23} (2017), 141--147.


\bibitem{Survey} L. Shapiro, A survey of the Riordan group, available electronically at
\href{http://www.combinatorics.cn/activities/Riordan\%20Group.pdf} {Center for Combinatorics}, Nankai University, 2018.

\bibitem{SGWW} L. W. Shapiro, S. Getu, W.-J. Woan, and L. C.
    Woodson,
The Riordan group, \emph{Discr. Appl. Math.} \textbf{34}
(1991), 229--239.

\bibitem{SL1} N. J. A.~Sloane, \emph{The
On-Line Encyclopedia of Integer Sequences}. Published electronically
at \href{http://oeis.org}{http://oeis.org}, 2019.

\bibitem{SL2} N. J. A.~Sloane, The On-Line Encyclopedia of Integer
Sequences, \emph{Notices Amer. Math. Soc.}, \textbf{50} (2003),  912--915.

\bibitem{Swart} C. S. Swart, \emph{Elliptic curves and related sequences}, PhD Thesis, University of London, 2003.

\bibitem{Poorten} Alfred J. van der Poorten,
Elliptic curves and continued fractions, {\it J. Integer
Sequences\/} {\bf 8} (2005), \href{http://www.cs.uwaterloo.ca/journals/JIS/VOL8/Poorten/vdp40.html}{Article 05.2.5}.

\bibitem{Viennot} Une th\'eorie combinatoire des polyn\^omes orthogonaux g\'en\'eraux, UQAM, Montreal, Quebec, 1983.

\bibitem{Wall} H.~S. Wall, \emph{Analytic Theory of
    Continued Fractions}, AMS Chelsea Publishing, 2001.

\bibitem{Xin} G. Xin, Proof of the Somos-$4$ Hankel determinants conjecture, \emph{Adv. in Appl. Math.}, \textbf{42} (2009), 152--156.

\bibitem{Yura} F. Yura, Hankel determinant solution for elliptic sequences, preprint 2014. Available at \url{https://arxiv.org/abs/1411.6972}.

\end{thebibliography}
\end{document}